\documentclass[leqno, a4paper]{amsart}
\usepackage{amsmath, amssymb, amsthm, mathrsfs, a4wide}
\usepackage[usenames]{color}
\usepackage{eepic,epic}
\usepackage{tcolorbox}
\usepackage{xcolor}

\newtheorem{theorem}{Theorem}[section]

\newtheorem{corollary}[theorem]{Corollary}
\newtheorem{lemma}[theorem]{Lemma}
\newtheorem{proposition}[theorem]{Proposition}

\theoremstyle{definition}
\newtheorem{definition}[theorem]{Definition}

\theoremstyle{remark}
\newtheorem{remark}[theorem]{Remark}
\newtheorem{example}[theorem]{Example}

\newtheorem*{claim}{Claim}

\numberwithin{equation}{section}

\newcommand{\op}{\operatorname}

\newcommand{\brak}[1]{\langle #1 \rangle}

\newcommand{\bH}{\mathbb{H}}
\newcommand{\N}{\mathbb{N}}

\newcommand{\R}{\mathbb{R}}

\newcommand{\Z}{\mathbb{Z}}

\newcommand{\cE}{\mathcal{E}}

\newcommand{\cH}{\mathcal{H}}
\newcommand{\cL}{\mathcal{L}}
\newcommand{\cU}{\mathcal{U}}

\newcommand{\cX}{\mathcal{X}}

\newcommand{\sN}{\mathscr{N}}

\newcommand{\fg}{\mathfrak{g}}

\newcommand{\rk}{\op{rk}}

\newcommand{\ad}{\op{ad}}

\newcommand{\Ow}{\op{O}_{w}}

\begin{document}

\title{Privileged Coordinates and Nilpotent Approximation for Carnot Manifolds, II. Carnot Coordinates}

\subjclass[2000]{Primary}

\author{Woocheol Choi}
\address{Department of Mathematics Education, Incheon National University, Incheon, South Korea}
\email{choiwc@inu.ac.kr}

 \author{Rapha\"el Ponge}
\address{Department of Mathematical Sciences, Seoul National University, Seoul, South Korea}
 \email{ponge.snu@gmail.com}

 \thanks{WC was partially supported by POSCO TJ Park Foundation. RP\ was partially supported by Research Resettlement Fund and Foreign Faculty Research Fund of Seoul National University, and  Basic Research grants 2013R1A1A2008802 and 2016R1D1A1B01015971 of National Research Foundation of Korea.}

\begin{abstract}
This paper is a sequel of~\cite{CP:Privileged} and deals with privileged coordinates and nilpotent approximation of Carnot manifolds. By a Carnot manifold it is meant a manifold equipped with a  filtration by subbundles of the tangent bundle which is 
 compatible with the Lie bracket of vector fields. In this paper, we single out a special class of privileged coordinates in which the nilpotent approximation at a given point of a Carnot manifold is given by its tangent group.  We call these coordinates Carnot coordinates. Examples of Carnot coordinates include Darboux coordinates on contact manifolds and the canonical coordinates of the first kind of Goodman and Rothschild-Stein. By converting the privileged coordinate of Bella\"iche into Carnot coordinates we obtain an effective construction of Carnot coordinates, which we call $\varepsilon$-Carnot coordinates. They form the building block of all systems of Carnot coordinates. On a graded nilpotent Lie group they are given by the group law of the group. For general Carnot manifolds, they depend smoothly on the base point. Moreover, in Carnot coordinates at a given point, they are osculated in a very precise manner by the group law of the tangent group at the point. 
\end{abstract}

\maketitle

 \section{Introduction}
 This paper is part of a series of two papers on privileged coordinates and nilpotent approximation of Carnot manifolds. By a Carnot manifold we mean a manifold $M$ together with a filtration of subbundles,
\begin{equation}
   H_{1}\subset H_{2}\subset \cdots \subset H_{r}=TM,
   \label{eq:Intro.Carnot-filtration}
\end{equation}which is compatible with the Lie bracket of vector fields. 
We refer to~\cite{CP:Privileged}, and the references therein, for various examples of Carnot manifolds. 
It is a general understanding that (graded) nilpotent Lie group are models for Carnot manifolds. 
From an algebraic perspective, any filtration~(\ref{eq:Intro.Carnot-filtration}) gives rise to a graded vector bundle $\fg M:=\fg_1M\oplus \cdots \oplus \fg_r M$, where $\fg_wM=H_{w}/H_{w-1}$. As a vector bundle $\fg M$ is (locally) isomorphic to the tangent bundle $TM$. Moreover, as observed by Tanaka~\cite{Ta:JMKU70}, the Lie bracket of vector fields induces on each fiber $\fg M(a)$, $a\in M$, a Lie algebra bracket which turns $\fg M(a)$ into a graded nilpotent Lie algebra. Equipping it with its Dynkin product we obtain a graded nilpotent group $GM(a)$, which is called the tangent group at $a$ (see Section~\ref{sec:tangent-group} for a review of this construction). 

There is an alternative construction of nilpotent graded groups associated with $(M,H)$.  This construction originated from the work of Folland-Sein~\cite{FS:CMPAM74} and  Rothschild-Stein~\cite{RS:ActaMath76} on hypoelliptic PDEs. In this context, it is natural to weight the differentiation by a vector field according to which sub-bundle $H_j$ of the filtration~(\ref{eq:Intro.Carnot-filtration}) that vector field lies. For instance, as $[H_1,H_1]\subset H_2$ we would like to regard directions in $H_2\setminus H_1$ as having order~$2$. More generally, directions in $H_w\setminus H_{w-1}$ have weight~$w$. Note that this notion of weight is consistent with the grading of $\fg M$ described above. 

In local coordinates centered at a given point $a\in M$ this gives rise to a one-parameter of anisotropic dilations $\delta_t$, $t\in \R$. Rescaling vector fields by means of these dilations and letting $t\rightarrow 0$ we obtain ``anisotropic asymptotic expansions" whose leading terms form a graded nilpotent Lie algebra of vector fields $\fg^{(a)}$. As it turns out, the algebraic structure of $\fg^{(a)}$ heavily depends on the choice of the coordinates. However, there is a special class of coordinates, called privileged coordinates, where the grading of $\fg^{(a)}$ is compatible with the weight. In addition, in these coordinates the graded nilpotent Lie algebra $\fg^{(a)}$ is isomorphic to $\fg M(a)$. The Lie algebra $\fg^{(a)}$ is actually the Lie algebra of left-invariant vector fields on a graded nilpotent Lie group $G^{(a)}$. This group gives rise to the so-called nilpotent approximation of $(M,H)$ at $a$. We refer to Section~\ref{sec:Privileged} and Section~\ref{sec:Nilpotent-approximation}, and the references therein, for more details on privileged coordinates and the nilpotent approximation of Carnot manifolds. 

At the conceptual level and for the sake of applications, it is desirable to understand better the relationship between the tangent group and the nilpotent approximation. The aim of this paper is to single out a special class of privileged coordinates, called Carnot coordinates, for which the nilpotent approximation is naturally given by the tangent group. These coordinates are an important ingredient in~\cite{CP:Groupoid} for obtaining a generalization of Pansu derivative to maps between general Carnot manifolds and for constructing an analogue for Carnot manifolds of Connes' tangent groupoid. The existence of such a groupoid was conjectured by Bella\"iche~\cite{Be:Tangent}. 

The tangent groupoid of a Carnot manifold provides us with definitive evidence that the tangent groups described above are the relevant osculating objects for Carnot manifolds. In particular, this gives a conceptual explanation for the occurrence of the group structure. This also gives an osculation result in a differentiable fashion. On Carnot-Caratheodory manifolds privileged coordinates allows us to have an osculation result in the metric sense (see~\cite{Be:Tangent, Mi:JDG85, NSW:AM85}). However, this approach does not account for the group structure. Therefore, by leading us to the Carnot groupoid Carnot coordinates allows us to get a more precise osculation result. 

The nilpotent approximation $G^{(a)}$ at a given point $a\in M$ arises from anisotropic asymptotic expansions of vector fields in privileged coordinates. This depends on the choice of the privileged coordinates. We refer to~\cite{CP:Privileged} and Section~\ref{sec:Nilpotent-approximation} for an algebraic characterization of the nilpotent Lie groups that arises as nilpotent approximation at $a$. It can be shown that suitable tangent frames (called $H$-frames; see Definition~\ref{def:Carnot-mflds.H-frames}) naturally identify the tangent group $GM(a)$ with a graded nilpotent Lie algebra that satisfies this algebraic characterization (see Proposition~\ref{prop:G(a)-sNX(a)}). Therefore, it gives rise to the nilpotent approximation at $a$ in suitable systems of privileged coordinates, which we call Carnot coordinates. For instance, Darboux coordinates on a contact manifold are Carnot coordinates (Proposition~\ref{prop:Carnot-coor.Darboux-coord}). 

By using the results of~\cite{CP:Privileged} we can show that any system of privileged coordinates can be converted into a system of Carnot coordinates by means of a unique homogeneous change of variables (Theorem~\ref{thm:Carnot-coord.w-hom-privileged-to-Carnot}). Here homogeneity is meant with respect to the anisotropic dilations associated with the filtration~(\ref{eq:Intro.Carnot-filtration}). 
We also characterize the changes of coordinates that transform a given system of Carnot coordinates into another system of Carnot coordinates (Proposition~\ref{prop:char-Carnot-coord.phi-x-Ow}). 
As an application of this characterization, we can show that the canonical coordinates of the first kind of Goodman~\cite{Go:LNM76} and Rothschild-Stein~\cite{RS:ActaMath76} are Carnot coordinates (Proposition~\ref{prop:Canonical-coord.1st-kind}). 
In contrast, when the filtration~(\ref{eq:Intro.Carnot-filtration}) has length $r\geq 2$, the canonical coordinates of the second kind of Bianchini-Stefani~\cite{BS:SIAMJCO90} and Hermes~\cite{He:SIAMR91} are never Carnot coordinates (Proposition~\ref{prop:Canonical-coord.2nd-kind}). 
 
In~\cite{Be:Tangent} Bella\"iche gave an effective construction of privileged coordinates on Carnot-Carath\'eodory manifolds by means of a suitable polynomial change of coordinates. This was extended to Carnot manifolds in~\cite{CP:Privileged}. 
By converting these privileged coordinates into Carnot coordinates by means of the homogeneous change of variables mentioned above we then obtain an effective construction of Carnot coordinates (Proposition~\ref{prop:Carnot-coord.carnot-cor}). We call these coordinates $\varepsilon$-coordinates. 
Combining this construction  with the characterization result mentioned above then allows us to get all the systems of Carnot coordinates at a given point (Corollary~\ref{cor:vareps-Carnot.all-Carnot-coord}). In particular, this shows that the $\varepsilon$-coordinates form the building block of all systems of Carnot coordinates at a given point. 
 
The $\varepsilon$-Carnot coordinates at a given point $a$ are obtained by means of a polynomial change of variables $x\rightarrow \varepsilon_a(x)$ associated with the datum of an $H$-frame near $a$. We can show that the polynomial map $\varepsilon_a(x)$ is the only one of its form that produces Carnot coordinates (Theorem~\ref{thm:Carnot-coord.unicity-normal2}). 
One application of this result is the description of the $\varepsilon$-Carnot coordinates on a graded nilpotent Lie group in terms of the group law (see Proposition~\ref{prop-static}). 
Furthermore, the coefficients of the  polynomial map $\varepsilon_a(x)$ and its inverse map $\varepsilon_a^{-1}(a)$  are universal polynomials in the partial derivatives at $a$ of the coefficients of the vector fields of the $H$-frame (see Proposition~\ref{prop:Carnot-prop.polynomial-X} and Proposition~\ref{polynomial-XI} for the precise statements). As an immediate consequence, these maps both depend smoothly on the base point $a$ (Corollary~\ref{prop-carnot-sm}). 

Finally, we show that, at least asymptotically, there is a close relationship between $\varepsilon$-Carnot coordinates and the group laws of the tangent groups. More precisely, in Carnot coordinates at a point $a$, the maps $(x,y)\rightarrow \varepsilon_x(y)$ and   $(x,y)\rightarrow \varepsilon_x^{-1}(y)$ are osculated in very precise manner by the group law maps $(x,y)\rightarrow x^{-1}\cdot y$ and $(x,y)\rightarrow x\cdot y$, where $\cdot$ is the group law of $GM(a)$ under the identification mentioned above (see Proposition~\ref{prop-com}). Obtaining these osculation results was an important impetus for the research of this paper and its prequel~\cite{CP:Privileged}. These results are important ingredients in the construction of the tangent groupoid of a manifold in~\cite{CP:Groupoid} and the construction of a full symbolic calculus for hypoelliptic pseudodifferential operators on Carnot manifolds in~\cite{CP:Carnot-calculus}.

 This paper is organized as follows. In Section~\ref{sec:Carnot-manifolds}, we present the main definitions regarding Carnot manifolds, including the  
construction of the tangent group bundle of a Carnot manifold. In Section~\ref{sec:Privileged}, we review some important facts on privileged coordinates. 
In Section~\ref{sec:Nilpotent-approximation}, we summarize the main results of~\cite{CP:Privileged} on the nilpotent approximation of Carnot manifolds. In Section~\ref{sec:Carnot-coord.existence}, we introduce Carnot coordinates, establish their properties, and exhibit a few examples. In Section~\ref{sec:Canonical-coordinate}, we look at the canonical coordinates of the 1st kind of~\cite{Go:LNM76, RS:ActaMath76} and canonical coordinates of the 2nd kind of~\cite{BS:SIAMJCO90, He:SIAMR91}. In Section~\ref{sec:vareps-Carnot-coord}, we construct and characterize the $\varepsilon$-Carnot coordinates and look at some examples, including $\varepsilon$-Carnot coordinates on graded nilpotent Lie groups. In Section~\ref{sec:H-frame-dependence}, we look at the dependence of the map $\varepsilon_a$ and its inverse map with respect to the $H$-frame they are associated with, and we show that in Carnot coordinates they are osculated by the group law of the tangent group. 
 
\subsection*{Acknowledgements}
The authors wish to thank Andrei Agrachev, Davide Barilari, Enrico Le Donne, and Fr\'ed\'eric Jean for useful discussions related to
the subject matter of this paper. They also thank anonymous referees whose insightful comments help improving the presentation of the paper. In addition, they would like to thank Henri Poincar\'e Institute (Paris, France), McGill University 
(Montr\'eal, Canada) and University of California at Berkeley (Berkeley, USA) for their hospitality during the 
preparation of this paper.  

\section{Carnot Manifolds}\label{sec:Carnot-manifolds}
In this section, we briefly review the main facts about Carnot manifolds. A more complete account can be found in~\cite{CP:Privileged} and the references therein. 

\subsection{Carnot Groups and Graded nilpotent Lie groups} 
In what follows, by the Lie algebra of a Lie group we shall mean the tangent space at the unit element equipped with Lie bracket induced by the Lie bracket of left-invariant vector fields. 

\begin{definition} \label{def:Carnot.Carnot-algebra}
A \emph{step $r$ nilpotent graded Lie algebra} is the data of a real Lie algebra $(\fg,[\cdot, \cdot])$ and a grading $  \fg=\fg_{1}\oplus \fg_{2}\oplus \cdots \oplus \fg_{r}$, which is compatible with the Lie bracket, i.e., 
\begin{equation}
        [\fg_{w},\fg_{w'}]\subset \fg_{w+w'} \ 
        \text{for $w+w'\leq r$} \quad \text{and} \quad [\fg_{w},\fg_{w'}]=\{0\}\ 
        \text{for $w+w'> r$}.
         \label{eq:Carnot.grading-bracket}
\end{equation}
We further say that $\fg$ is a \emph{Carnot algebra} when $\fg_{w+1}=[\fg_{1},\fg_{w}]$ for $w=1,\ldots, r-1$. 
A \emph{graded nilpotent Lie group} (resp., \emph{Carnot group}) is a connected  and simply connected nilpotent real Lie group whose Lie algebra is a  graded nilpotent  Lie algebra (resp., Carnot algebra). 
\end{definition}

Let $G$ be a step $r$ graded nilpotent Lie group with unit $e$. Its Lie algebra $\fg=T_eG$ comes equipped with a grading $  \fg=\fg_{1}\oplus \fg_{2}\oplus \cdots \oplus \fg_{r}$, which is compatible with its Lie algebra bracket. This grading  gives rise to a family of anisotropic dilations $\xi \rightarrow t\cdot \xi$, $t \in \R\setminus 0$, which are automorphisms of $\fg$ given by
\begin{equation}
 t\cdot (\xi_1+\xi_2 +\cdots + \xi_r) = t\xi_1+ t^2\xi_2 +\cdots + t^r \xi_r, \qquad \xi_j \in \fg_j.  
 \label{eq:Carnot.dilations}
\end{equation}
In addition, $\fg$ is canonically isomorphic to the Lie algebra of left-invariant vector fields on $G$: to any $\xi \in \fg$ corresponds the unique left-invariant vector field $X_\xi$ on $G$ such that $X_\xi(e)=\xi$. The flow $\exp(tX_\xi)$ exists for all $t\in \R$, 
and so we have a globally defined exponential map $\exp:\fg \rightarrow G$ which is a diffeomorphism given by
\begin{equation*}
\exp(\xi)=\exp(X_\xi), \qquad \text{where $\exp(X_\xi):=\exp(tX_\xi)(e)_{|t=1}$}. 
\end{equation*}

For $\xi \in \fg$, let $\ad_\xi: \fg\rightarrow \fg$ be the adjoint endomorphism of $\xi$, i.e., $\ad_\xi \eta =[\xi,\eta]$ for all $\eta \in \fg$. This is a nilpotent endomorphism.  
By  the Baker-Campbell-Hausdorff formula we have 
\begin{equation}
\exp(\xi) \exp(\eta)= \exp(\xi \cdot \eta) \qquad \text{for all $\xi,\eta \in \fg$},
\label{eq:Carnot.BCH-Formula}
\end{equation}where $\xi \cdot \eta$ is given by the Dynkin product, 
\begin{align}
 \xi \cdot \eta & = \sum_{n\geq 1} \frac{(-1)^{n+1}}{n} \sum_{\substack{\alpha, \beta \in \N_0^n\\ \alpha_j+\beta_j\geq 1}} 
 \frac{(|\alpha|+|\beta|)^{-1}}{\alpha!\beta!} (\ad_\xi)^{\alpha_1} (\ad_\eta)^{\beta_1} \cdots (\ad_\xi)^{\alpha_n} (\ad_\eta)^{\beta_n-1}\eta \nonumber \\
 & = \xi +\eta + \frac{1}{2}[\xi,\eta] + \frac{1}{12} \left( [\xi,[\xi,\eta]]+  [\eta,[\eta,\xi]]\right) - \frac{1}{24}[\eta,[\xi,[\xi,\eta]]] + \cdots .
 \label{eq:Carnot.Dynkin-product}
\end{align}
The above summations are finite, since all the iterated brackets of length~$\geq r+1$ are zero. Any Lie algebra automorphism of $\fg$ then lifts to a Lie group isomorphism of $G$. In particular, the dilations~(\ref{eq:Carnot.dilations}) give rise to Lie group isomorphisms $\delta_t:G\rightarrow G$, $t\in \R\setminus 0$.

Conversely, if $\fg$ is a graded nilpotent Lie algebra, then~(\ref{eq:Carnot.Dynkin-product}) defines a product on $\fg$. This turns $\fg$ into a Lie group with unit $0$. Under the identification $\fg\simeq T_0 \fg$, the corresponding Lie algebra is naturally identified with $\fg$, and so we obtain a graded nilpotent Lie group. Under this identification 
the exponential map becomes the identity  map. Moreover, inversion with respect to the group law~(\ref{eq:Carnot.Dynkin-product}) is given by
\begin{equation}
\xi^{-1}=-\xi \qquad \text{for all $\xi \in \fg$}. 
\label{eq:GM.inverse}
\end{equation}

\subsection{Carnot manifolds} 
In what follows, given distributions $H_{j}\subset TM$, $j=1,2$,  on a manifold $M$, we shall denote by $[H_{1},H_{2}]$ 
the distribution generated by the Lie brackets of their sections, i.e., 
\begin{equation*}
  [H_{1},H_{2}]=\bigsqcup_{x\in M}\biggl\{ [X_{1},X_{2}](x); \ X_{j}\in C^{\infty}(M,H_{j}), j=1,2\biggr\}.
\end{equation*}

\begin{definition}
    A \emph{Carnot manifold} is a pair $(M,H)$, where $M$ is a manifold and $H=(H_{1},\ldots,H_{r})$ is a finite filtration of 
    subbundles, 
    \begin{equation}
       H_{1}\subset H_{2}\subset \cdots \subset H_{r}=TM,  
        \label{eq:Carnot.Carnot-filtration}
    \end{equation}which is compatible with the Lie bracket of vector fields, i.e., 
    \begin{equation}
        [H_{w},H_{w'}]\subset H_{w+w'}\qquad 
        \text{for $w+w'\leq r$}.
        \label{eq:Carnot-mflds.bracket-condition}
    \end{equation}
    The number $r$ is called the \emph{step} of the Carnot manifold $(M,H)$. The sequence $(\rk H_{1},\ldots, \rk H_{r})$ is called its \emph{type}.  
\end{definition}

\begin{remark}
We refer to~\cite{CP:Privileged}, and the references therein, for various examples of Carnot manifolds. 
Many of those examples are equiregular Carnot-Carath\'eodory manifolds in the sense of~\cite{Gr:CC}. In this case the Carnot filtration~(\ref{eq:Intro.Carnot-filtration}) arises from the iterated  Lie bracket of sections of $H_1$. However,  even for studying equiregular (and even non-regular) Carnot-Carath\'eodory structures we may be naturally led to consider non bracket-generated Carnot filtrations (see~\cite{CP:Privileged}). 
\end{remark}

Let $(M^{n},H)$ be an $n$-dimensional Carnot manifold of step $r$, so that $H=(H_{1},\ldots,H_{r})$, 
where the subbundles $H_{j}$ satisfy~(\ref{eq:Carnot.Carnot-filtration}).  

\begin{definition}\label{def:Carnot.weight-sequence}
    The \emph{weight sequence} of a Carnot manifold $(M,H)$ is the sequence $w=(w_{1},\ldots,w_{n})$ defined by
    \begin{equation}
    w_{j}=\min\{w\in \{1,\ldots,r\}; j\leq \rk H_{w}\}.
    \label{eq:Carnot.weight}
\end{equation}
\end{definition}

\begin{remark}
    Two Carnot manifolds have same type if and only if they have same weight sequence. 
\end{remark}

Throughout this paper we will make use of the following type of tangent frames.  

\begin{definition}\label{def:Carnot-mflds.H-frames}  
An \emph{$H$-frame} over an open $U\subset M$ is a tangent frame $(X_{1},\ldots,X_{n})$ over $U$ which is compatible with the 
filtration $(H_{1},\ldots.,H_{r})$ in the sense that, for $w=1,\ldots,r$, the vector fields $X_{j}$, $w_{j}= w$, are sections of $H_{w}$.
\end{definition}

\begin{remark}\label{rmk:Carnot-mfld.local-frames}
    If $(X_{1},\ldots,X_{n})$ is an $H$-frame near a point $a\in M$, then, for $w=1,\ldots,r$, the family $\{X_{j}; \ w_{j}\leq w\}$ is a local frame 
of $H_{w}$ near the point $a$. Note this implies that $X_{j}$ is a section of $H_{w_{j}}\setminus H_{w_{j}-1}$ when $w_{j}\geq 2$. 
\end{remark}

\begin{remark}\label{rmk:Carnot-mfld.brackets-H-frame}
    Let $(X_{1},\ldots,X_{n})$ be an $H$-frame near a point $a\in M$. As explained in~\cite{CP:Privileged},  the condition~(\ref{eq:Carnot-mflds.bracket-condition}) implies that, near $x=a$, there are smooth functions $L_{ij}^{k}(x)$, $w_k\leq w_i+w_j$,  such that, for  $i,j=1,\ldots, n$, we have 
    \begin{equation}
        [X_{i},X_{j}](x)=\sum_{w_{k}\leq w_{i}+w_{j}}L_{ij}^{k}(x)X_{k}(x) \qquad \textup{near $x=a$}. 
        \label{eq:Carnot-mfld.brackets-H-frame}
    \end{equation}   
\end{remark}

\subsection{The Tangent Group Bundle of a Carnot Manifold}\label{sec:tangent-group}
The Carnot filtration $H=(H_{1},\ldots,H_{r})$ gives rise to a graded vector bundle defined as follows. For $w=1,\ldots,r$, set 
$\fg_{w}M=H_{w}/H_{w-1}$ (with the convention that $H_{0}=\{0\}$), and define
\begin{equation}\label{eq-grading}
    \fg M := \fg_{1}M\oplus \cdots \oplus \fg_{r}M.
\end{equation}
In what follows, we denote by $\times_M$ the fiber products of vector bundles over $M$. 

\begin{lemma}[\cite{Ta:JMKU70}]\label{lem-dep}
 The Lie bracket of vector fields induces smooth bilinear bundle maps, 
\begin{equation*}
    \cL_{w,w'}:\fg_{w}M\times_M \fg_{w'}M\longrightarrow \fg_{w+w'}M, \qquad w+w'\leq r. 
\end{equation*}More precisely, given any $a\in M$ and any 
 section $X$ (resp., $Y$) of $H_{w}$ (resp., $H_{w'}$) near $a$, if we let $\xi(a)$ (resp., $\eta(a)$)  be the class of $X(a)$ (resp., $Y(a)$) in $\fg_{w}M(a)$ 
 (resp., $\fg_{w'}M(a)$), then we have
\begin{equation}
 \cL_{w,w'}\left(\xi(a),\eta(a)\right)= \textup{class of  $[X,Y](a)$ in $\fg_{w+w'}M(a)$}. 
  \label{eq:Tangent.Lie-backet}
\end{equation}
 \end{lemma}

\begin{remark}\label{rmk:Tangent-group.definition-cL}
 The proof of Lemma~\ref{lem-dep} is based on the following observation. Given $a\in M$ and local sections $X_j$, $j=1,2$, of $H_{w_j}$ near $x=a$, the class of $[X_1,X_2](a)$ modulo $H_{w_1}(a)+H_{w_2}(a)$ only depends on the values of $X_1(a)$ and $X_2(a)$. Therefore, if $w_i+w_j\leq r$, then the condition~(\ref{eq:Carnot.Carnot-filtration}) ensures us that the class of $[X_1,X_2](a)$ in $\fg_{w_1+w_2}M(a)$ only depends on $X_1(a)$ and $X_2(a)$. 
\end{remark}

 \begin{remark}\label{rmk:Tangent-group.frame-fgM}
    Let $(X_{1},\ldots,X_{n})$ be an $H$-frame over an open set $U\subset M$. For $j=1,\ldots,n$ and $x\in U$, let  
     $\xi_{j}(x)$ be the class of $X_{j}(x)$ in $\fg_{w_{j}}M(x)$. Then, as mentioned in~\cite{CP:Privileged}, for $w=1,\ldots,r$, 
     the family $\{\xi_{j}; w_{j}=w\}$ is a smooth frame of $\fg_{w}M$ over $U$, and so $(\xi_{1},\ldots,\xi_{n})$ is a smooth
     frame of $\fg M$ over $U$ which is compatible with the grading~(\ref{eq-grading}). We shall call such a frame a \emph{graded frame}. 
     In particular, for every $x\in U$, we obtain a \emph{graded basis} of $\fg M(x)$. 
       In addition, by using~(\ref{eq:Carnot-mfld.brackets-H-frame}) and~(\ref{eq:Tangent.Lie-backet}) 
     we get
     \begin{equation}
         \cL_{w_{i},w_{j}}(\xi_{i}(x),\xi_{j}(x))=\sum_{w_{k}=w_{i}+w_{j}}L_{ij}^{k}(x)\xi_{k}(x) \qquad 
         \text{for $w_{i}+w_{j}\leq r$}, 
         \label{eq:Tangent-Group.Levi-H-frame}
     \end{equation}
    where the coefficients $L_{ij}^k(x)$, $w_i+w_j=w_k$, are defined by~(\ref{eq:Carnot-mfld.brackets-H-frame}). 
 \end{remark}

\begin{definition}
The bilinear bundle map $[\cdot,\cdot]:\fg M\times_M  \fg M\rightarrow \fg M$ is defined as follows. For $a\in M$ and $(\xi,\eta) \in \fg_wM(a)\times \fg_{w'}M(a)$ we set
    \begin{equation}\label{eq-brak-a}
        [\xi,\eta]=\left\{
        \begin{array}{ll}
            \cL_{w,w'}(a)(\xi,\eta)& \text{if $w+w'\leq r$},  \smallskip \\
           0 &  \text{if $w+w'> r.$ }
        \end{array}\right.
    \end{equation}
 We then extend $[\cdot,\cdot]$ to all $\fg M\times_M  \fg M$ by bilinearity.   
\end{definition}

Lemma~\ref{lem-dep} ensures us that $[\cdot, \cdot]$ is a smooth bilinear bundle map. Furthermore, on each fiber  $\fg M(a)$, $a\in M$, it defines 
a Lie algebra bracket such that
       \begin{align}
                   \left[ \fg_{w}M, \fg_{w'}M \right]  \subset \fg_{w+w'}M   & \quad \text{if $w+w'\leq r$}, \\
                \left[\fg_{w}M,\fg_{w'}M\right] =\{0\}  &  \quad \text{if $w+w'>r $}.
          \label{eq:Carnot.Carnot-grading}
           \end{align}
Therefore, the Lie bracket is compatible with the grading~(\ref{eq-grading}), and so it turns $\fg M(a)$ into a (graded) nilpotent  Lie algebra of step $r$.  
 
\begin{proposition}[\cite{Ta:JMKU70}]
   $(\fg M, [\cdot, \cdot])$ is a smooth bundle of step $r$ graded nilpotent Lie algebras.
\end{proposition}

\begin{definition}
The Lie algebra bundle $(\fg M, [\cdot, \cdot])$  is called the \emph{tangent Lie algebra bundle} of $(M,H)$.
\end{definition}

\begin{remark}\label{rem-xixj}
    Let $(X_{1},\ldots,X_{n})$ be an $H$-frame near a point $a\in M$. As mentioned above, this gives rise to a 
    frame $(\xi_{1},\ldots,\xi_{n})$ of $\fg M$ near $x=a$, where $\xi_{j}$ is the class of $X_{j}$ in 
    $\fg_{w_{j}}M$. Furthermore, it follows from~(\ref{eq:Tangent.Lie-backet}) and~(\ref{eq-brak-a}) that, near $x=a$, we have
    \begin{equation}\label{eq-jk}
        [\xi_{i}(x),\xi_{j}(x)]=\left\{
        \begin{array}{cl}
          {\displaystyle \sum_{w_{i}+w_{j}=w_k}L_{ij}^{k}(x)\xi_{k}(x)}    & \text{if $w_{i}+w_{j}\leq r$},  \\
            0  & \text{if $w_{i}+w_{j}>r$},
        \end{array}\right.
    \end{equation}where the functions $L_{ij}^{k}(x)$ are given by~(\ref{eq:Carnot-mfld.brackets-H-frame}). Specializing this to $x=a$ provides us with the structure constants 
    of $\fg M(a)$ with respect to the basis $(\xi_{1}(a),\ldots,\xi_{n}(a))$.  
\end{remark}

The Lie algebra bundle $\fg M$ gives rise to a Lie group bundle $GM$ as follows. As a manifold we take $GM$ to be $\fg M$ and we equip the fibers $GM(a)=\fg M(a)$ with the Dynkin product~(\ref{eq:Carnot.Dynkin-product}).This turns $GM(a)$ into a step~$r$ graded nilpotent Lie group with unit $0$ whose Lie algebra is naturally isomorphic to $\fg M(a)$. The fiberwise product on $GM$ is smooth, and so we obtain the following result. 
 
\begin{proposition}
    $GM$ is a smooth bundle of step $r$ graded nilpotent Lie groups. 
\end{proposition}

\begin{definition}\label{def:Tangent.Tangent-group}
    $GM$ is called the \emph{tangent group bundle} of $(M,H)$. Each fiber $GM(a)$, $a\in M$ is called the \emph{tangent group} of $(M,H)$ at $a$. 
\end{definition}

\subsection{The Tangent group bundle of a graded nilpotent Lie group} 
We end this section with the description of the tangent group bundle of a graded nilpotent Lie group. 

Let $G$ be a step $r$ graded nilpotent Lie group $G$ with unit $e$. The Lie algebra $\fg=TG(e)$ has a grading $\fg=\fg_1\oplus \cdots \oplus \fg_r$ satisfying~(\ref{eq:Carnot.grading-bracket}). In what follows, given any $\xi \in \fg$, we denote by $X_\xi$ the unique left-invariant vector on $G$ such that $X_\xi(e)=\xi$. Let $\lambda: G\times \fg \rightarrow TG$ be the vector bundle map given by 
\begin{equation*}
\lambda(a,\xi)= X_\xi(a)= \lambda_a'(0)\xi\in TG(a)  \qquad \text{for all $(a,\xi)\in G\times \fg$},
\end{equation*}where $\lambda_a:G\rightarrow G$ is the left-multiplication by $a$. Then $\lambda$ is a vector bundle isomorphism, and so it provides us with a global trivialization $G\times \fg \simeq TG$. Moreover, this isomorphism is equivariant with respect to the left-regular actions of $G$ on $TG$ and $G\times \fg$. 

For $w=1,\ldots, r$, set  $E_w=\lambda( G\times \fg_w)$, i.e., $E_w$ is the $G$-bundle over $G$ obtained by left-translation of $\fg_w$.  In particular, by $G$-equivariance the grading $\fg=\fg_1\oplus \cdots \oplus \fg_r$ gives rise to the $G$-vector bundle grading $TG=E_1 \oplus \cdots \oplus E_r$. This grading then gives rise to the filtration $H_1\subset \cdots \subset H_r=TG$, where $H_w:=E_1 \oplus \cdots \oplus E_w$. 

\begin{lemma}\label{lem:tangent.graded-group}
 The filtration $(H_1,\ldots, H_r)$ is compatible with the Lie bracket of vector fields, i.e.,  $[H_w,H_{w'}]\subset H_{w+w'}$ whenever $w+w'\leq r$. 
\end{lemma}
\begin{proof}
It is enough to show that $[E_w,E_{w'}]\subset H_w$ whenever $w+w'\leq r$. Given $a \in G$, let $X$ (resp., $X'$) be a local section of $E_w$ (resp., $E_{w'}$) near $a$. In addition, let $\xi \in \fg_w$ and $\xi'\in \fg_{w'}$ be such that $X(a)=X_\xi(a)$ and $X'(a)=X_{\xi'}(a)$. In view of Remark~\ref{rmk:Tangent-group.definition-cL} we have 
\begin{align}
 [X,X'](a) & = [X_\xi, X_{\xi'}](a) \ \bmod E_{w}(a)+E_{w'}(a)\nonumber\\
 &=  X_{[\xi, \xi']}(a) \ \bmod H_{w+w'-1}(a). 
 \label{eq:Tangent.Lie-bracket-graded-group}
\end{align}
As~(\ref{eq:Carnot.grading-bracket}) ensures us that $[\xi, \xi']\in \fg_{w+w'}$, we see that $ X_{[\xi, \xi']}(a)\in E_{w+w'}$, and so  $[X,X'](a)$ is contained in $H_{w+w'}(a)$. This shows that $[E_w,E_{w'}]\subset H_{w+w'}$. The proof is complete. 
\end{proof}

The filtration $H=(H_1, \ldots, H_r)$ defines a left-invariant Carnot manifold structure on $G$. This Carnot manifold structure is uniquely determined by the grading of $\fg$. We shall denote by $\fg G$ and $GG$ the tangent Lie algebra bundle and tangent group bundle of $(G,H)$. We observe that the grading $TG=E_1\oplus \cdots \oplus E_r$ provides us with a natural graded vector bundle identification $TG\simeq \fg G$. Composing it with the trivialization $\lambda$ above, we then 
obtain an isomorphism of \emph{graded} vector bundles $\hat{\lambda}: G\times \fg \xrightarrow{\sim} \fg G$. Moreover, it follows from~(\ref{eq:Tangent.Lie-bracket-graded-group}) and the definition of the Lie bracket in the fibers of $\fg G$ that this isomorphism is compatible with the fiberwise Lie bracket. Therefore, we actually obtain an isomorphism of graded nilpotent Lie algebra bundles. This allows us to identify the tangent Lie algebra bundle of $G$ with the trivial graded nilpotent Lie algebra bundle $G \times \fg$. By functoriality we obtain an isomorphism of graded nilpotent Lie group bundles, 
\begin{equation*}
\hat{\Lambda}: G\times G \stackrel{\sim}{\longrightarrow} GG. 
\end{equation*}
This allows us to identify the tangent group bundle of $G$ with the trivial graded nilpotent Lie group bundle $G\times G$.

\section{Privileged Coordinates}\label{sec:Privileged}
Throughout the rest of the paper, we let $(M,H)$ be an $n$-dimensional Carnot manifold of step $r$, so that 
$H=(H_{1},\ldots,H_{r})$ is a filtration of subbundles satisfying~(\ref{eq:Carnot.Carnot-filtration}). We also let 
$w=(w_{1},\ldots,w_{n})$ be the weight sequence of $(M,H)$ (\emph{cf.}\ Definition~\ref{def:Carnot.weight-sequence}). 

In this section, we gather the main facts regarding anistropic asymptotic analysis and privileged coordinates. We refer to~\cite{CP:Privileged} and 
the references therein for a most extensive account.  

\subsection{Anisotropic Asymptotic Analysis}
In what follows, we let $\delta_t:\R^n\rightarrow \R^n$, $t\in \R$, be the one-parameter group of anisotropic dilations defined by 
\begin{equation}
    \delta_{t}(x)=t\cdot x:=(t^{w_{1}}x_{1},\ldots,t^{w_n}x_{n}), \qquad t\in \R, \ x\in \R^n.
    \label{eq:Nilpotent.dilations2}
\end{equation}

We shall say that a function $f(x)$ on $\R^{n}$ is  \emph{homogeneous} of degree $w$,
$w\in \Z$, with respect to the dilations~(\ref{eq:Nilpotent.dilations2}) when 
\begin{equation}
f(t\cdot x)=t^{w}f(x) \qquad \text{for all $x\in \R^n$ and $t\in \R$}. 
\label{eq:priv.homog-function}
\end{equation}
For instance, given any multi-order $\alpha \in \N_{0}^{n}$, the monomial $x^{\alpha}$ is homogeneous of degree $\brak \alpha$. 
Note also that any smooth homogeneous function must be a polynomial.

In what follows we let $U$ be an open neighborhood of the origin $0\in \R^{n}$. 

\begin{definition}
Let $f\in C^{\infty}(U)$ and $w\in \N_0$. We shall say that
\begin{enumerate}
    \item  $f$ has weight $\geq w$ when $\partial^{\alpha}_{x}f(0)=0$ for all multi-orders
    $\alpha \in \N_{0}^{n}$ such that $\brak\alpha<w$.
    \item  $f$ has weight $w$ when $f(x)$ has weight~$\geq w$ and there is a multi-order
    $\alpha\in \N_{0}^{n}$ with $\brak\alpha=w$ such that $\partial^{\alpha}_{x}f(0)\neq 0$.
\end{enumerate}
\end{definition}

\begin{example}
 Let $f(x)\in C^\infty(\R^n)$ be homogeneous of degree $w$ with respect to the dilations~(\ref{eq:Nilpotent.dilations2}). 
 If $f(x)$ is not zero everywhere, then $f(x)$ has weight $w$. 
\end{example}

\begin{definition}\label{def:anisotropic-pseudo-norm}
A \emph{pseudo-norm} on $\R^n$ is any continuous function $\|\cdot \|:\R^n \rightarrow \R$ which is (strictly) positive on $\R^n\setminus 0$ and satisfies 
\begin{equation}
 \|t\cdot x\|= |t| \|x\| \qquad \text{for all $x\in \R^n $ and $t\in \R$}.
 \label{eq:anisotropic.homogeneity-pseudo-norm}
\end{equation}
\end{definition}
 
\begin{example}
The following function on $\R^n$ is a pseudo-norm, 
\begin{equation}
\|x\|_1=  |x_1|^{\frac{1}{w_1}} + \cdots + |x_n|^{\frac{1}{w_n}}, \qquad x\in \R^n.
\label{eq:anistropic-pseudo-norm1}
\end{equation}
\end{example}

\begin{remark}
 It can be shown that all pseudo-norms are equivalent (see, e.g., \cite{CP:Privileged}).  
\end{remark}

From now on, we let $\|\cdot \|$ be a pseudo-norm on $\R^n$. For instance, we may take the pseudo-norm~(\ref{eq:anistropic-pseudo-norm1}). 
\begin{definition}[\cite{CP:Privileged}]\label{def:anisotropic.Thetaw}
Let $\Theta(x)=(\Theta_{1}(x),\ldots,\Theta_{n'}(x))$ be a smooth map from $U$  to $\R^{n}$. Given $m\in \Z$, $m\geq -w_{n}$, we say that $\Theta(x)$ is $\Ow(\|x\|^{w+m})$, and write $\Theta(x)=\Ow(\|x\|^{w+m})$, when, for $k=1,\ldots,n$, we have
\begin{equation*}
 \Theta_{k}(x)=\op{O}(\|x\|^{w_{k}'+m}) \qquad \text{near $x=0$}.  
\end{equation*}
\end{definition}

We have the following characterization of $\Ow(\|x\|^{w+m})$-maps.
\begin{lemma}[see~\cite{CP:Privileged}]\label{lem-eq-we} Let $\Theta(x)=(\Theta_{1}(x),\ldots,\Theta_{n'}(x))$ be a smooth map from $U$ to $\R^{n'}$, and set $ \cU=\{(x,t)\in U\times \R ; \ t\cdot x \in U\}$. Given any $m\in \N_0$, the following are equivalent:
\begin{enumerate}
 \item[(i)]    The map $\Theta(x)$ is $\Ow(\|x\|^{w+m})$ near $x=0$.

 \item[(ii)]   For $k=1,\ldots,n'$, the component $\Theta_{k}(x)$ has weight~$\geq w_{k}'+m$. 

 \item[(iii)] For all $x\in \R^{n}$ and as $t\rightarrow 0$, we have $t^{-1}\cdot \Theta(t\cdot x)=\op{O}(t^{m})$.  

 
 \item[(iv)]  There is $\tilde{\Theta}(x,t)\in C^\infty(\cU, \R^{n'})$ such that 
 $t^{-1}\cdot \Theta(t\cdot x)=t^m \tilde{\Theta}(x,t)$ for all $(x,t)\in \cU$, $t\neq 0$.
\end{enumerate}
\end{lemma}

In the special case of diffeomorphisms we further have the following result. 

\begin{proposition}[see~\cite{CP:Privileged}]\label{lem:Carnot-coord.inverse-Ow}
Let $\phi:U_{1}\rightarrow U_{2}$ be a smooth diffeomorphism, where $U_{1}$ and $U_{2}$ are open neighborhoods of 
the origin $0\in \R^{n}$. Assume further that $\phi(0)=0$ and there is $m\in \N$ such that, near $x=0$, we have 
\begin{equation}
    \phi(x)=\hat{\phi}(x)+\Ow(\|x\|^{w+m}),
    \label{eq:Carnot-coord.phi-hatphi}
\end{equation}where $\hat{\phi}(x)$ is a $w$-homogeneous polynomial map. Then $\hat{\phi}(x)$ is a diffeomorphism of $\R^n$ and, near 
$x=0$, we have
\begin{equation*}
    \phi^{-1}(x)=\hat{\phi}^{-1}(x)+\Ow(\|x\|^{w+m}). 
\end{equation*}
\end{proposition}

The dilations~(\ref{eq:Nilpotent.dilations2}) act on vector fields by pullback. If $X=\sum_{j=1}^n a_j(x)\partial_{x_j}$ is a vector field on $U$, then we have 
\begin{equation}
 \delta_t^*X = \sum_{j=1}^n t^{-w_j}a_j(t\cdot x)\partial_{x_j} \qquad \text{for all $t\in \R^*$}. 
 \label{eq:anisotropic-dtX}
\end{equation}
We then shall say that a vector field $X$ on $\R^{n}$  is \emph{homogeneous} of degree $w$, $w\in \Z$, when
\begin{equation}
    \delta_{t}^{*}X=t^{\omega} X \qquad \text{for all $t\in \R^*$}. 
    \label{eq:priv.homog-vecto-field}
\end{equation}
For instance, the vector field $\partial_{x_j}$, $j=1,\ldots, n$, is homogeneous of degree $-w_j$. 

\begin{definition} 
Let $X=\sum_j a_j(x)\partial_j$ be a smooth vector fields on $U$. We say that $X$ has weight $w$, $w\in \Z$, when 
\begin{enumerate}
    \item[(i)]  Each coefficient $a_j(x)$ has weight~$\geq w+w_j$.

      \item[(ii)]  There is at least one coefficient $a_{j}(x)$ that has weight~$w+w_j$.
  \end{enumerate}
\end{definition}

Let $\cX(U)$ be the space of (smooth) vector fields on $U$. We endow it with its standard Fr\'echet space topology; this allows us to speak about asymptotic expansions in $\cX(U)$ (see~\cite{CP:Privileged}). We have the following asymptotic expansion result.

\begin{proposition}[see~\cite{CP:Privileged}]\label{prop:nilp-approx.vector-fields}
    Let $X$ be a smooth vector field on $U$. Then, as $t\rightarrow 0$, we have 
                \begin{equation}\label{eq-delta-X}
            \delta_{t}^{*}X\simeq \sum_{\ell\geq -w_n} t^{\ell}X^{[\ell]} \qquad \text{in $\cX(U)$},
        \end{equation}
        where $X^{[\ell]}$ is a homogeneous polynomial vector field of degree $\ell$. In particular, 
$X$ has weight $w$ if and only if, as $t\rightarrow 0$, we have 
       \begin{equation*}
       \delta_{t}^{*}X= t^{w}X^{[w]} +\op{O}(t^{w+1})\qquad \text{in $\cX(U)$}.
    \end{equation*}
\end{proposition}

\subsection{Construction of privileged coordinates} 
From now on we work in local coordinates near a given point $a\in M$ around which there is an $H$-frame 
$(X_{1},\ldots,X_{n})$. 

\begin{definition}\label{def-lin-adp}
   We say that local coordinates $(x_{1},\ldots,x_{n})$ centered at a point $a\in M$ are \emph{linearly adapted} to
   the $H$-frame $X_{1},\ldots,X_{n}$ when $X_{j}({0})=\partial_{j}$ for $j=1,\ldots,n$.
\end{definition}

\begin{lemma}[\cite{Be:Tangent, CP:Privileged}]\label{lem-affine}
    Given local coordinates $x=(x_{1},\cdots,x_{n})$, there is a unique affine change of coordinates $x\rightarrow T_{a}(x)$
   which provides us with local coordinates centered at $a$ that are linearly adapted to the $H$-frame $(X_{1},\cdots X_{n})$.
\end{lemma}

In what follows, given any
finite sequence $I=(i_{1},\ldots,i_{k})$ with values in $\{1,\ldots,n\}$, we set
$X_{I}:=X_{i_{1}}\cdots X_{i_{k}}$ and  $\brak I=w_{i_{1}}+\cdots + w_{i_{k}}$. Given any multi-order $\alpha\in \N_{0}^{n}$, we also set
$ \langle \alpha \rangle = w_1 \alpha_1 + \cdots + w_n \alpha_n$.

\begin{definition}[\cite{Be:Tangent, CP:Privileged}] \label{orderD} 
Let $f(x)$ be a smooth function defined near $x=a$ and $N$ a non-negative integer. We say that $f(x)$ has \emph{order} $N$ at $a$ when the following two conditions are satisfied:
 \begin{enumerate}
     \item  $X_{I}f(a)=0$ whenever $\brak I<N$.

     \item There is a sequence $I=(i_{1},\ldots,i_{k})$ taken values in $\{1,\ldots,n\}$ with $\brak I=N$ such that $X_{I}f(a)\neq 0$.
 \end{enumerate}
\end{definition}

\begin{remark}
 The order of a function is independent of the choice of the $H$-frame $(X_1,\ldots, X_n)$ (see~\cite{Be:Tangent, CP:Privileged}). In this sense this is an intrinsic notion. 
\end{remark}

\begin{definition}[\cite{Be:Tangent, BS:SIAMJCO90, CP:Privileged}]\label{def:Privileged-coordinates}
    We say that local coordinates $x=(x_{1},\ldots,x_{n})$ centered at $a$ are \emph{privileged coordinates} at $a$ adapted to
    the $H$-frame ($X_{1},\ldots, X_{n})$ when the following two conditions are satisfied:
    \begin{enumerate}
        \item[(i)] These coordinates are linearly adapted at $a$ to the $H$-frame $(X_{1},\ldots ,X_{n})$. 

        \item[(ii)] For all $j=1,\ldots,n$, the coordinate function $x_{j}$ has order $w_{j}$ at $a$.
    \end{enumerate}
\end{definition}
 
The following shows how to convert linearly adapted coordinates into privileged coordinates. 

\begin{proposition}[\cite{Be:Tangent, CP:Privileged}]\label{prop:privileged}
Let $(x_{1},\ldots,x_{n})$ be local coordinates centered at $a$ that are linearly adapted to the $H$-frame
$(X_{1},\ldots,X_{n})$. Then there is a unique change of coordinates $x\rightarrow \psi(x)$ such that
\begin{enumerate}
    \item  It provides us with privileged coordinates at $a$.

    \item  For $k=1,\ldots,n$, the $k$-th component $\psi_{k}(x)$ is of the form,
    \begin{equation}
        \psi_{k}(x)=x_{k}+\sum_{\substack{\brak\alpha <w_{k}\\ |\alpha|\geq 2}} a_{k\alpha}x^{\alpha}, \qquad
        a_{k\alpha}\in \R.
        \label{eq-form-h}
    \end{equation}
\end{enumerate}
\end{proposition}

\begin{remark}\label{rem-ab}
    The proofs of Proposition~\ref{prop:privileged} in~\cite{Be:Tangent, CP:Privileged} actually produce an effective algorithm to compute the coefficients  $a_{k\alpha}$,  in~(\ref{eq-form-h}) . More precisely, we have the following recursive relations,
     \begin{equation}\label{eq-pri-det}
    \alpha!a_{k\alpha}=- \left.X^{\alpha}(x_{k})\right|_{x=0}
    - \sum_{\substack{\brak\beta <w_{k}\\
 2\leq |\beta|<|\alpha|}} a_{k\beta} \left.X^{\alpha}(x^{\beta})\right|_{x=0}, \quad 2\leq |\alpha|\leq \brak\alpha <w_k. 
\end{equation}
In particular, we have $\alpha!a_{k\alpha}=- \left.X^{\alpha}(x_{k})\right|_{x=0}$ when $w_k=3$ and $|\alpha|=\brak\alpha =2$. It follows from this      that each coefficient $a_{k\alpha}$ is a universal polynomial in the 
    derivatives $\left.X^{\alpha}(x^{\beta})\right|_{x=0}$ with $\brak \beta \leq w_{k}$ and $|\beta|\geq 1$. Set 
    $X_{j}=\sum_{k=1}^{n}b_{jk}(x)\partial_{x_k}$. An induction shows that
    \begin{equation*}
        X^{\alpha}=\sum_{1\leq |\beta|\leq |\alpha|}b_{\alpha\beta}(x)\partial^{\beta}_x,
    \end{equation*}where $b_{\alpha\beta}(x)$ is a universal polynomial in the partial derivatives 
    $\partial^{\gamma}b_{jk}(x)$ with $|\gamma|\leq |\alpha|-|\beta|$. As 
    $\left.X^{\alpha}(x^{\beta})\right|_{x=0}=\beta! b_{\alpha\beta}(0)$, we then deduce that each coefficient 
    $a_{j\alpha}$ is a universal polynomial in the partial derivatives $\partial^{\gamma}b_{kl}(0)$ with $|\gamma|\leq 
    |\alpha|-1$. 
\end{remark}

\begin{definition}\label{def-psia}
The map $\psi_a:\R^n\rightarrow \R^n$ is composition $\hat{\psi}_a\circ T_a$, where  $T_{a}$ is the affine map from Lemma~\ref{lem-affine} and  
$\hat{\psi}_a$ the polynomial diffeomorphism associated by Proposition~\ref{prop:privileged} with the linearly adapted  coordinates provided by $T_a$.  
\end{definition}

Proposition~\ref{prop:privileged} and the definition of $\psi_a$ immediately imply the following statement. 

\begin{proposition}\label{prop:privileged.psi-privileged}
 The change of variables $x\rightarrow \psi_a(x)$ provides us with privileged coordinates at $a$ that are adapted to the $H$-frame $(X_1,\ldots, X_n)$.
\end{proposition}
 
We shall refer to the coordinates provided by Proposition~\ref{prop:privileged.psi-privileged} as the \emph{$\psi$-privileged coordinates}. 

\begin{proposition}[\cite{CP:Privileged}]\label{prop-uni}
   The $\psi$-privileged coordinates are the unique privileged coordinates at $a$ adapted to the $H$-frame
    $(X_{1},\ldots,X_{n})$ that are given by a change of coordinates of the form $y=\hat{\psi}(Tx)$, where $T$ is an affine
    map such that $T(a)=0$ and $\hat{\psi}(x)$ is a polynomial diffeomorphism of the form~(\ref{eq-form-h}).
\end{proposition}

\begin{remark}
 We refer to~\cite{AGS:AAM89, AS:DANSSSR87, St:1986} for alternative polynomial constructions of privileged coordinates. 
  \end{remark}

\begin{remark}\label{rmk:privileged.canonical-coord}
Examples of non-polynomial  privileged coordinates are provided by the canonical coordinates of the first kind of Goodman~\cite{Go:LNM76} and 
Rothschild-Stein~\cite{RS:ActaMath76} and the canonical coordinates of the second kind of Bianchini-Stefani~\cite{BS:SIAMJCO90} and  Hermes~\cite{He:SIAMR91}.  The former coordinates are given by the inverse of the local diffeomorphism, 
\begin{equation}
 (x_1,\ldots, x_n)\longrightarrow \exp\left( x_1X_1+\cdots + x_n X_n\right)\!(a).
 \label{eq:privileged.canonical-coord1} 
\end{equation}
The canonical coordinates of the second kind arise from the inverse of the local diffeomorphism,
\begin{equation}
 (x_1,\ldots, x_n)\longrightarrow \exp\left( x_1X_1\right) \circ \cdots  \circ \exp\left( x_n X_n\right)\!(a). 
  \label{eq:privileged.canonical-coord2}  
\end{equation}
It is shown in~\cite{BS:SIAMJCO90} that the canonical coordinates of the 2nd kind are privileged coordinates in the sense of Definition~\ref{def:Privileged-coordinates} (see also~\cite{CP:Privileged, Je:Brief14, Mo:AMS02}). For the canonical coordinates of the 1st kind the result is proved in~\cite{Je:Brief14} (see also~\cite{CP:Privileged}). 
 \end{remark}

\begin{remark}
 On Carnot-Carath\'eodory manifolds privileged coordinates are the very coordinates in which the ball-box theorem holds~\cite{NSW:AM85}. 
\end{remark}

\subsection{Characterization of privileged coordinates} 
We have the following characterization of privileged coordinates. 

\begin{proposition}[\cite{CP:Privileged}] \label{prop-pri-equiv}
    Let $(x_{1},\ldots,x_{n})$ be local coordinates centered at $a$ that are linearly adapted to the $H$-frame
    $(X_{1},\ldots, X_{n})$.  In addition, let $U\subset \R^n$ be the range of these coordinates.  Then the following are equivalent:
    \begin{enumerate}
        \item[(i)]  The local coordinates $(x_{1},\ldots,x_{n})$ are privileged coordinates at $a$.
        
        \item[(ii)] For $j=1,\ldots, n$ and as $t\rightarrow 0$, we have
                       \begin{equation}
                              t^{w_j} \delta_t^*X_j = X_j^{(a)} + \op{O}(t) \qquad \text{in $\cX(U)$}, 
                              \label{eq:char-priv.model-vector-field2}
                       \end{equation}
                     where $X_j^{(a)}$ is  homogeneous of degree~$-w_j$. 
        
        \item[(iii)]  For $j=1,\ldots,n$, the vector field $X_{j}$ has weight $-w_{j}$ in the local coordinates $(x_{1},\ldots,x_{n})$.
    \end{enumerate}
\end{proposition}

\begin{remark}[see, e.g.,~\cite{CP:Privileged}] 
  As the coordinates $(x_{1},\ldots,x_{n})$ are linearly adapted at $a$ to the $(X_{1},\ldots,X_{n})$, in these coordinates, for $j=1,\ldots, n$, we can write, 
\begin{equation}
 X_j= \partial_j + \sum_{1\leq k \leq n}b_{jk}(x)\partial_{x_k}, \qquad b_{jk}(x)\in C^\infty(U),\ b_{jk}(0)=0. 
 \label{eq:Nilpotent.Xj-linearly-adapted}
\end{equation}
The model vector fields $X_{j}^{(a)}$, $j=1,\ldots, n$, are then given by
\begin{equation}\label{eq-homo-app}
    X_{j}^{(a)}= \partial_{x_j}+ \sum_{\substack{w_{j}+\brak\alpha=w_{k}\\ w_k>w_j}} \frac{1}{\alpha!}\partial^{\alpha}_x
    b_{jk}(0)x^{\alpha} \partial_{x_k}.
\end{equation}
\end{remark}

\begin{remark}\label{rem:privileged-coordinates-equivalence}
Goodman~\cite{Go:LNM76} defined privileged coordinates as linearly adapted local coordinates satisfying~(\ref{eq:char-priv.model-vector-field2}). It is well known that privileged coordinates in the sense of Definition~\ref{def:Privileged-coordinates} are privileged coordinates in Goodman's sense (see~\cite{BS:SIAMJCO90, Be:Tangent, Je:Brief14}). The converse property is the main insight of Proposition~\ref{prop-pri-equiv}. There is a number of privileged coordinates in Goodman's sense (see, e.g.,~\cite{ABB:SRG, AGS:AAM89, BS:SIAMJCO90, Go:LNM76, Gr:CC, He:SIAMR91, MM:JAM00, Me:CPDE76, Mo:AMS02, RS:ActaMath76, St:1986}). Therefore, all these  coordinates are privileged coordinates in the sense of Definition~\ref{def:Privileged-coordinates}. 
\end{remark}

One consequence of Proposition~\ref{prop-pri-equiv} in~\cite{CP:Privileged} is the description of \emph{all} the systems of privileged coordinates at a given point.
 
\begin{proposition}[\cite{CP:Privileged}]\label{prop:char-priv.phi-hatphi-Ow}
Suppose that $(x_1,\ldots, x_n)$ are privileged coordinates at $a$ adapted to the $H$-frame $(X_1,\ldots, X_n)$. Then a change of coordinates $x\rightarrow \phi(x)$ produces privileged coordinates at $a$ adapted to $(X_1,\ldots, X_n)$ if and only if we have 
\begin{equation}
 \phi(x)=\hat{\phi}(x) +\Ow\left(\|x\|^{w+1}\right)  \qquad \text{near $x=0$},
 \label{eq:char-priv-coord.phi-x-Ow}
\end{equation}
where $\hat{\phi}(x)$ is a $w$-homogeneous polynomial diffeomorphism such that $\hat{\phi}'(0)=\op{id}$. 
\end{proposition}

\begin{remark}[see~\cite{CP:Privileged}]
 By combining Proposition~\ref{prop:char-priv.phi-hatphi-Ow} with Proposition~\ref{prop-uni} we see that system of local coordinates is a system of privileged coordinates at $a$ adapted to $(X_1,\ldots, X_n)$ if and only if it arises from a local chart of the form $\phi \circ \psi_{\kappa(a)} \circ  \kappa$, where $\kappa$ is  local chart near $a$ such that $\kappa(a)=0$, the map $\psi_{\kappa(a)}$ is as in Definition~\ref{def-psia}, and $\phi$ is a diffeomorphism near the origin $0\in \R^n$ satisfying~(\ref{eq:char-priv-coord.phi-x-Ow}).  
\end{remark}

\section{Nilpotent Approximation of a Carnot Manifold}\label{sec:Nilpotent-approximation} 
In this section, we recall the main facts on the nilpotent approximation of a Carnot manifold. Throughout this section we let $(X_{1},\ldots.,X_{n})$ be a $H$-frame near a given point $a\in M$. 

\subsection{Nilpotent approximation} 
Suppose that $(x_1,\ldots, x_n)$ are privileged coordinates at $a$ adapted to $(X_1,\ldots, X_n)$. By Proposition~\ref{prop-pri-equiv}, for $j=1,\ldots, n$, the vector field $X_j$ has weight $-w_j$ and we have an asymptotic of the form~(\ref{eq:char-priv.model-vector-field2}). We shall call the leading vector field $X_j^{(a)}$ in~(\ref{eq:char-priv.model-vector-field2}) the \emph{model vector field} of $X_j$. 

Let $\tilde{\fg}^{(a)}$ the subspace of $T\R^{n}$ spanned by the model vector fields $X^{(a)}_{1}, \ldots, X^{(a)}_{n}$. We have a natural grading, 
\begin{equation}\label{eq-g-grading}
    \tilde{\fg}^{(a)}=\fg_{1}^{(a)}\oplus \cdots \oplus \fg_{r}^{(a)},
\end{equation}
where $\fg_{w}^{(a)}$ is the subspace spanned by the vector fields $X_j^{(a)}$ with $w_j=w$. Moreover, as $(X_{1},\ldots,X_{n})$ is an $H$-frame, it follows from Remark~\ref{rmk:Carnot-mfld.brackets-H-frame} 
that there are smooth functions $L_{ij}^{k}(x)$, $w_{k}\leq w_{i}+w_{j}$, satisfying~(\ref{eq:Carnot-mfld.brackets-H-frame}). We then have the relations, 
 \begin{equation}
     [X_{i}^{(a)}, X_{j}^{(a)}]=\left\{
     \begin{array}{cl}
   {\displaystyle \sum_{w_{k}=w_{i}+w_{j}}L_{ij}^{k}(a)X_{k}^{(a)}} & \text{if $w_{i}+w_{j}\leq  r$},   \\
         0 & \text{otherwise}.  \\
     \end{array}\right.
     \label{eq:Nilpotent.structure-constants-Xj(a)}
 \end{equation}
This shows that $\tilde{\fg}^{(a)}$ is closed under the Lie bracket of vector fields and this Lie bracket is compatible with the grading~(\ref{eq-g-grading}). Thus, $\tilde{\fg}^{(a)}$ is a graded nilpotent Lie algebra of vector fields on $\R^n$. Furthermore, using~(\ref{eq-jk}) we see that the Lie algebras $\tilde{\fg}^{(a)}$ and $\fg M(a)$ are isomorphic. In addition, the homogeneity of the model vector fields $X_j^{(a)}$ implies that the dilations $\delta_t$, $t\in \R$, give rise to a Lie algebra automorphisms of $\tilde{\fg}^{(a)}$. 

We realize $\tilde{\fg}^{(a)}$ as a Lie algebra of left-invariant vector fields on a Lie group $G^{(a)}$ as follows. 
Let $U$ be the range of the privileged coordinates $(x_1,\ldots, x_n)$. As there are linearly adapted at $a$ to $(X_1,\ldots, X_n)$,  we know by~(\ref{eq:Nilpotent.Xj-linearly-adapted}) that in these coordinates $X_j=\partial_{x_j} + \sum_{k=1}^n b_{jk}(x) \partial_{x_k}$, with $b_{jk}(x)\in C^\infty(U)$ such that $b_{jk}(0)=0$. The formula~(\ref{eq-homo-app}) then expresses each model vector field $X_j^{(a)}$, $j=1,\ldots, n$, in terms of the partial derivatives $\partial^\alpha b_{jk}(0)$ with $\brak{\alpha}+w_j=w_k$ and $w_k>w_j$. 

Let  $X=\sum \xi_j X_j^{(a)}$, $\xi_j\in \R$, be a vector field in $\tilde{\fg}^{(a)}$.
 As mentioned in~\cite{CP:Privileged}, for any given 
$y\in \R^n$, the flow $x(t)=\exp(tX)(y)$ is solution of the following \emph{triangular} first order differential system, 
   \begin{equation*}
       \dot{x}_{k}(t)  =\xi_{k}+ \sum_{\substack{w_{j}+\brak\alpha=w_{k}\\
     w_k>w_j}} \xi_{k} b_{jk\alpha} \prod_{w_l<w_k} x_l(t)^{\alpha_l},       \qquad k=1,\ldots,n,  
     \label{eq:Carnot-coord.ODE-system}
   \end{equation*}
where we have set $b_{jk\alpha}=(\alpha!)^{-1}\partial^\alpha b_{jk}(0)$. As this differential system is triangular system it can be solved recursively. Using the initial condition $x(0)=y$ we get 
\begin{gather}
 x_k(t)=y_k +t\xi_k \qquad \text{if $w_k=1$},
 \label{eq:nilpotent.Gammakalphbeta-recursive1}\\
 x_k(t)= y_k + t\xi_k +  \sum_{\substack{w_{j}+\brak\alpha=w_{k}\\
    w_k>w_j}} \xi_{j} b_{jk\alpha} \int_0^t\prod_{w_l<w_k}x_l(s)^{\alpha_l}ds \qquad \text{if $w_k\geq 2$}. 
    \label{eq:nilpotent.Gammakalphbeta-recursive2}
\end{gather}

\begin{lemma}[\cite{CP:Privileged}] \label{lem:nilpotent.flowX}
 Let $X=\sum \xi_j X_j^{(a)}$ and $y\in \R^n$ be as above. Then the flow $x(t):=\exp(tX)(y)$ exists for all $t\in \R$. Moreover, it takes the form, 
\begin{equation}
 x_k(t)= y_k+ t\xi_k  + \sum_{\substack{\brak\alpha+\brak\beta=w_{k}\\ |\alpha|+|\beta|\geq 2}} \hat{c}_{k\alpha\beta} y^\alpha(t\xi)^\beta , \qquad k=1,\ldots, n.  
\label{eq:Carnot-coord.xk(t)}
\end{equation}
where $\hat{c}_{k\alpha\beta}$ is a universal polynomial $\hat{\Gamma}_{k\alpha\beta}(\partial^\gamma b_{jl}(0))$ in the partial derivatives $\partial^\gamma b_{jl}(0)$ with 
$w_j+\brak\gamma=w_l\leq w_k$ and $w_j\leq w_l$. Each polynomial $\hat{\Gamma}_{k\alpha\beta}$ is 
determined recursively from the polynomials $\hat{\Gamma}_{j\gamma \delta}$ with $w_j<w_k$ by using~(\ref{eq:nilpotent.Gammakalphbeta-recursive1})--(\ref{eq:nilpotent.Gammakalphbeta-recursive2}).
\end{lemma}

It follows from Lemma~\ref{lem:nilpotent.flowX}  that we have a globally defined smooth exponential map $\exp: \tilde{\fg}^{(a)} \rightarrow \R^n$ given by 
\begin{equation}\label{eq-CH}
    \exp(X):=\left.\exp(tX)(0)\right|_{t=1} \qquad \text{for all $X\in \tilde{\fg}^{(a)}$}. 
\end{equation}
Although, $\exp(X)$ \emph{a priori} arises from the solution of an ODE system, it follows from~(\ref{eq:Carnot-coord.xk(t)}) that $\exp(X)$ is determined effectively in terms of  the coordinates of $X$ in the basis $(X_1^{(a)}, \ldots, X_n^{(a)})$ and the coefficients of the model vector fields $X_j^{(a)}$. Indeed, if  $X= \sum \xi_j X_j^{(a)}$, $\xi_j\in \R$, then setting $t=1$ and $y=0$ shows that $x=\exp(X)$ is given by
\begin{equation}
 x_k= \xi_k  +   \sum_{\substack{\brak\alpha=w_{k}\\ |\alpha|\geq 2}} \hat{c}_{k\alpha} \xi^\alpha ,\qquad k=1,\ldots, n, 
\label{eq:Carnot-coord.xk(1)}
\end{equation}
where we have set $\hat{c}_{k\alpha}=\hat{c}_{k0\alpha}$. It also follows from this formula that $\exp: \tilde{\fg}^{(a)}\rightarrow \R^n$ is a diffeomorphism, since it expresses $x=\exp(X)$ as a triangular polynomial map in the coordinates $(\xi_1,\ldots, \xi_n)$ the diagonal of which is the identity map. In addition, we observe that $x_k$, as a polynomial in $\xi$, is homogeneous of degree $w_k$ with respect to the dilations~(\ref{eq:Nilpotent.dilations2}). Thus, for all $t\in \R^*$, we have 
\begin{equation}
 t\cdot \exp(X)= \exp\left( \sum (t\cdot\xi)_j X_j^{(a)}\right) = \exp \left( \sum \xi_j t^{w_j} X_j^{(a)}\right)= \exp\left( \delta_{t^{-1}}^*X\right).
\label{eq:fga-homogeneity-expX}
\end{equation}

We define the Lie group $G^{(a)}$ as $\R^n$ equipped with the group law given by
\begin{equation}
 x\cdot y = \exp(X\cdot Y), \qquad X,Y\in \tilde{\fg}^{(a)},
 \label{eq:fga.group-law-G(a)}
\end{equation}
where $X$ and $Y$ are the unique elements of $\tilde{\fg}^{(a)}$ such that $\exp(X)=x$ and $\exp(Y)=y$, and $X\cdot Y$ is the Dynkin 
product~(\ref{eq:Carnot.Dynkin-product}). For $j=1,\ldots, n$, the vector field $X_j^{(a)}$ generates a one-parameter subgroup $\exp(tX_j^{(a)})$, $t\in \R$, in $G^{(a)}$, and so this is a left-invariant vector field on $G^{(a)}$. As $(X_1^{(a)}, \ldots, X_n^{(a)})$ is a basis of $\tilde{\fg}^{(a)}$, we then arrive at the following statement. 

\begin{proposition}\label{prop:Nilpotent.G(a)-tildefg(a)}
 The Lie algebra of left-invariant vector fields on $G^{(a)}$ is precisely $\tilde{\fg}^{(a)}$. 
\end{proposition}

\begin{definition}\label{def-Ga}
The graded nilpotent Lie group $G^{(a)}$ equipped with its left-invariant Carnot manifold structure is called the \emph{nilpotent approximation} of $(M,H)$ at $a$ with respect to the privileged coordinates $(x_1,\ldots, x_n)$. 
\end{definition}

\begin{definition}
 $\fg(a)$ is the nilpotent Lie algebra obtained by equipping $T\R^{n}(0)$ with the Lie bracket given by
  \begin{equation}
     [\partial_{i}, \partial_{j}]=\left\{
     \begin{array}{cl}
   {\displaystyle \sum_{w_{k}=w_{i}+w_{j}}L_{ij}^{k}(a)\partial_{k}} & \text{if $w_{i}+w_{j}\leq  r$},   \\
         0 & \text{otherwise}  \\
     \end{array}\right.
     \label{eq:Nilpotent.structure-constants-fg(a)}
 \end{equation}
\end{definition}

\begin{remark}
 The Lie algebra $\fg(a)$ depends only on the structure constants $L_{ij}^k(a)$, $w_i+w_j=w_k$, and so it does not depend on the choice of the privileged coordinates $(x_1,\ldots, x_n)$. 
\end{remark}

\begin{remark}
$\fg(a)$ is a graded nilpotent Lie algebra with respect to the grading,
\begin{equation}
 \fg(a)=\fg_1(a)\oplus \cdots \oplus \fg_r(a), \qquad \text{where}\ \fg_w(a):=\op{Span}\{\partial_j;\ w_j=w\}. 
 \label{eq:Nilpotent.grading-fg(a)}
\end{equation}
\end{remark}

\begin{proposition}[\cite{CP:Privileged}] \label{prop:Nilpotent.G(a)-fg(a)-graded}
 The Lie algebra of $G^{(a)}$ is precisely $\fg(a)$. Moreover, the dilations~(\ref{eq:Nilpotent.dilations2}) are group automorphisms of $G^{(a)}$.  
\end{proposition}

\subsection{Getting all nilpotent approximations} As it turns out, the  nilpotent approximation is by no means canonical since it depends on the choice of privileged coordinates. In fact, as shown in~\cite{CP:Privileged} we have a large class of group laws on $\R^n$ that arises from nilpotent approximations at any given point $a$.

\begin{definition}\label{def:Nilpotent.class-cN(a)}
 $\sN_{X}(a)$ consists of nilpotent groups $G$ that are obtained by equipping $\R^n$ with a group law such that
 \begin{enumerate}
 \item[(i)] The dilations~(\ref{eq:Nilpotent.dilations2}) are group automorphisms of $G$. 
 
 \item[(ii)] The Lie algebra $TG(0)$ of $G$ is precisely $\fg(a)$. 
\end{enumerate}
\end{definition}

\begin{remark}
 The condition (i) automatically implies that the origin is the unit of $G$.
\end{remark}

\begin{remark}\label{rmk:Nilpotent.compatibility-dilations-cN(a)}
 The conditions (i)--(ii) also imply that the dilations $\delta_t$, $t\in \R$, induce a family of dilations $\delta_t'(0)$, $t\in \R$, on $TG(0)=\tilde{\fg}^{(a)}$. In the basis 
 $(\partial_1, \ldots, \partial_n)$ they are given by~(\ref{eq:Nilpotent.dilations2}), and so they agree with the dilations~(\ref{eq:Carnot.dilations}) defined by the grading~(\ref{eq:Nilpotent.grading-fg(a)}) of $\fg(a)$. 
\end{remark}

By Proposition~\ref{prop:Nilpotent.G(a)-fg(a)-graded} we know that, given any system of privileged coordinates at $a$ adapted to $(X_1,\ldots, X_n)$, the nilpotent approximation $G^{(a)}$ is in the class $\sN_{X}(a)$. We obtain the converse of this result as follows. 

Let $G$ be a nilpotent Lie group in the class $\sN_X(a)$. Denote by $\tilde{\fg}$ its Lie algebra of left-invariant vector fields. Then $\tilde{\fg}$ has a \emph{canonical basi}s $(Y_1,\ldots, Y_n)$, where $Y_j$ is the unique left-invariant vector field on $G$ such that $Y_j(0)=\partial_j$. As mentioned in~\cite{CP:Privileged}, this basis has the following properties:
\begin{enumerate}
\item[(i)] For $j=1,\ldots, n$, the vector field $Y_j$ is homogeneous of degree~$-w_j$ with respect to the dilations~(\ref{eq:Nilpotent.dilations2}) and agrees with $\partial_j$ at $x=0$. 

\item[(ii)] The vector fields $Y_1, \ldots, Y_n$ satisfy the commutator relations~(\ref{eq:Nilpotent.structure-constants-Xj(a)})
\end{enumerate}
Conversely, if $(Y_1,\ldots, Y_n)$ is a family of vector fields on $\R^n$ satisfying the properties (i)--(ii), then this is the canonical basis of left-invariant vector fields on a unique nilpotent Lie group in the class $\sN_X(a)$ (see~\cite{CP:Privileged}). 

As pointed out in~\cite{CP:Privileged}, the property~(i) ensures us that, for every $y\in \R^n$ and every vector field $Y=\sum \eta_j Y_j$, $\eta_j\in \R$, in $\tilde{\fg}$, the flow $\exp(tY)(y)$ exists for all times $t\in \R$ and is of the form~(\ref{eq:Carnot-coord.xk(t)}). As a result, we have a $w$-homogeneous polynomial diffeomorphism $\exp_Y:\R^n\rightarrow \R^n$ given by
\begin{equation}
 \exp_Y(x)= \exp(x_1Y_1+\cdots + x_nY_n) , \qquad x\in \R^n,   
 \label{eq:many.exp_Y}
\end{equation}
where $ \exp(x_1Y_1+\cdots + x_nY_n)$ is defined as in~(\ref{eq-CH}). 
In fact, $\exp_Y(x)$ is of the form~(\ref{eq:Carnot-coord.xk(1)}), and so $\exp_Y'(0)=\op{id}$.  Note also that $\exp_Y$ is just the exponential map $\exp:\tilde{\fg}\rightarrow G$ in the coordinates defined by the basis $(Y_1,\ldots,Y_n)$.  

In addition, we let $(x_1,\ldots, x_n)$ be privileged coordinates at $a$ adapted to $(X_1,\ldots, X_n)$. For $j=1,\ldots, n$, we also let $X_j^{(a)}$ be the model vector field at $a$ of $X_j$, and we denote by $\tilde{\fg}^{(a)}$ the Lie algebra generated by $X_1^{(a)}, \ldots, X^{(a)}_n$. By Proposition~\ref{prop:Nilpotent.G(a)-tildefg(a)} this is the Lie algebra of left-invariant vector fields on the nilpotent approximation $G^{(a)}$ in the privileged coordinates  $(x_1,\ldots, x_n)$.  Note also that $(X_1^{(a)}, \ldots, X_n^{(a)})$ is the canonical basis of $\fg^{(a)}$. Like in~(\ref{eq:many.exp_Y}) we have a $w$-homogeneous polynomial diffeomorphism $\exp_{X^{(a)}}:\R^n \rightarrow \R^n$ given by
\begin{equation}
 \exp_{X^{(a)}}(x)= \exp\left( x_1X^{(a)}_1+ \cdots + x_nX^{(a)}_n\right), \qquad x\in \R^n.
 \label{eq:many.exp_X(a)} 
\end{equation}
We then define the map $ \phi_Y:\R^n\rightarrow \R^n$ by 
\begin{equation}
 \phi_Y(x)=\exp_Y\circ \exp_{X^{(a)}}^{-1}(x) \qquad \text{for all $x\in \R^n$}. 
 \label{eq:Nilpotent.phiY}
\end{equation}
Note that $\phi_Y$  is a $w$-homogeneous diffeomorphism of $\R^n$ whose differential at $x=0$ is the identity map, since $\exp_Y $ and $\exp_{X^{(a)}}$ are both such maps.

\begin{lemma}[\cite{CP:Privileged}]
 The diffeomorphism $\phi_Y$ is a group isomorphism from $G^{(a)}$ onto $G$. Moreover, this is the unique $w$-homogeneous diffeomorphism of $\R^n$ such that
   $ (\phi_Y)_{*}X_{j}^{(a)}=Y_{j}$ for $j=1,\ldots,n$. 
\end{lemma}

\begin{remark}
 Specializing this result to $G=G^{(a)}$ shows that the identity map is the unique $w$-homogeneous diffeomorphism $\phi:\R^n\rightarrow \R^n$ such that $\phi_*X_j^{(a)}=X_j^{(a)}$ for $j=1,\ldots, n$. 
\end{remark}

\begin{proposition}[\cite{CP:Privileged}]\label{prop:Nilpotent.coord-G(a)G}
 A change of coordinates $x\rightarrow \phi(x)$ yields privileged coordinates at $a$ adapted to $(X_1,\ldots, X_n)$ in which the nilpotent approximation is $G$  if and only if, near $x=0$, we have
\begin{equation}
 \phi(x) =\phi_Y(x) + \Ow\left(\|x\|^{w+1}\right). 
 \label{eq:nilp-approx.phi-phiY-Ow} 
\end{equation}
In particular, $x\rightarrow \phi_Y(x)$ is the unique such change of coordinates which is $w$-homogeneous.  
\end{proposition}

Combining Proposition~\ref{prop:Nilpotent.coord-G(a)G} with Proposition~\ref{prop:Nilpotent.G(a)-fg(a)-graded} we get the following result. 

\begin{corollary}[\cite{CP:Privileged}]\label{cor:many.nilp-sNX(a)}
 Let $G$ be a graded nilpotent Lie group built out of $\R^n$. Then $G$ provides us with the nilpotent approximation of $(M,H)$ at $a$ in some privileged coordinates at $a$ adapted to $(X_1,\ldots, X_n)$ if and only if it belongs to the class $\sN_{X}(a)$. 
\end{corollary}

\begin{remark}
 We refer to~\cite{CP:Privileged} for the explicit construction of all the nilpotent approximations of step~2 Carnot manifolds. For general Carnot manifolds the construction also produces explicit families of nilpotent approximations (see~\cite{CP:Privileged}). In any case, we see that the nilpotent approximation encompasses a very large class of groups.  
\end{remark}

\section{Carnot Coordinates}\label{sec:Carnot-coord.existence}
In this section, we shall refine  the construction of privileged coordinates of Section~\ref{sec:Privileged} to get a system of privileged coordinates with respect to which
the nilpotent approximation $G^{(a)}$ is naturally identified with the tangent group $GM(a)$.  As a result, this will clarify the relationship between the nilpotent approximation and tangent group. This will also provide us with a class of privileged coordinates in which the nilpotent approximation enjoys some kind of canonical form. 

Throughout this section we let $(X_{1},\ldots,X_{n})$ be an $H$-frame near a given point $a\in M$.  

\subsection{Coordinate description of $GM(a)$} 
By Remark~\ref{rmk:Carnot-mfld.brackets-H-frame} near the point $a$ there are functions $L_{ij}^k(x)$, $w_i+w_j\leq w_k$, given by the commutator relations~(\ref{eq:Carnot-mfld.brackets-H-frame}). In addition, as mentioned in Remark~\ref{rmk:Tangent-group.frame-fgM}, the $H$-frame $(X_{1},\ldots,X_{n})$ gives rise to a graded basis $(\xi_1(a), \ldots, \xi_n(a))$ of $\fg M(a)$, where $\xi_j(a)$ is the class of $X_j(a)$ in $\fg_{w_j}M(a)$. As pointed out in Remark~\ref{rem-xixj}, this graded basis satisfies the relations~(\ref{eq:Nilpotent.structure-constants-Xj(a)}). 

Given elements $\xi=\sum x_{i}\xi_{i}(a)$ and $\eta=\sum y_{j}\xi_{j}(a)$ of 
${\fg}M(a)$, using~(\ref{eq-jk}) we get
\begin{equation*}
    \ad_{\xi}\eta=\sum_{1\leq i,j \leq n} x_{i}y_{j}[\xi_{i}(a),\xi_{j}(a)]=
   \sum_{w_{i}+w_{j}=w_{k}}x_{i}y_{j}L_{ij}^k(a)\xi_k(a).
\end{equation*}
Therefore, the matrix $A_{a}(x)=\left(A_{a}(\xi)_{kj}\right)_{1\leq k,j\leq n}$ of $\ad_{\xi}$ with respect to the basis 
$(\xi_{1}(a),\ldots,\xi_{n}(a))$ is given by
\begin{equation*}
    A_{a}(x)_{kj} =\left\{\begin{array}{cl} {\displaystyle \sum_{w_i +w_j= w_k} L_{ij}^{k}(a) x_i }&\textup{if $w_j<w_k$},\\
0 &\textup{otherwise.}\end{array}\right.
\end{equation*}
In particular, the matrix  $A_{a}(x)$ is nilpotent and lower-triangular. Combining this with~(\ref{eq:Carnot.Dynkin-product}) we see that, in the coordinates
associated with the basis $(\xi_{1}(a),\ldots,\xi_{n}(a))$,  the product of $GM(a)$ is given by
\begin{equation}
  x\cdot y  =  \sum_{n\geq 1} \frac{(-1)^{n+1}}{n} \sum_{\substack{\alpha, \beta \in \N_0^n\\ \alpha_j+\beta_j\geq 1}} 
 \frac{(|\alpha|+|\beta|)^{-1}}{\alpha!\beta!} A_a(x)^{\alpha_1} A_a(y)^{\beta_1} \cdots A_a(x)^{\alpha_n} A_a(y)^{\beta_n-1}y. 
 \label{eq:GM.group-law}
\end{equation}

When $w_{k}=1$ and $w_k=2$ we obtain
\begin{gather}
    (x\cdot y)_{k}=x_{k}+y_{k} \qquad (w_k=1),
    \label{eq:GM.Xj1}\\
     (x\cdot y)_{k}=x_{k}+y_{k}+\frac{1}{2}\sum_{1\leq k \leq n}A_{a}(x)_{kj}y_{j}
   = x_{j}+y_{j}+\frac{1}{2}\sum_{w_{i}+w_{j}=2}L_{ij}^{k}(a)x_{i}y_{j} \quad (w_k=2).
   \label{eq:GM.Xj2}
\end{gather}
In general,  we have
\begin{equation}
    (x\cdot y)_{k}=x_{k}+y_{k}+ \frac{1}{2} \sum_{w_i+w_j=w_k} L_{ij}^k(a)x_iy_j + \sum_{\substack{\brak\alpha+\brak\beta=w_{k}\\ |\alpha|+|\beta|\geq 3}}
    B_{\alpha\beta}^k\left(L(a)\right)x^{\alpha}y^{\beta},
    \label{eq:GM.group-law2}
\end{equation}where the coefficients $B_{\alpha\beta}^k(L(a))$ are universal polynomials in the structure constants $L_{pq}^{l}(a)$, $w_p+w_q=w_l$. 

\begin{definition}
 The graded nilpotent Lie group $G(a)$ is $\R^n$ equipped with the group law~(\ref{eq:GM.group-law}). 
\end{definition}

\begin{remark}\label{rmk:Carnot-coord.independence-G(a)}
 As~(\ref{eq:GM.Xj1})--(\ref{eq:GM.group-law2}) show, the group law of $G(a)$ only depends on the structure constants $L_{ij}^k(a)$ with $w_i+w_j=w_k$. In particular,   $G(a)$ only depends on the jets of the $H$-frame $(X_1,\ldots,X_n)$ at $x=a$. 
\end{remark}

We summarize the previous discussion in the following. 

\begin{proposition}\label{prop:Carnot-coord.GM(a)-G(a)}
 Let $(\xi_1(a), \ldots, \xi_n(a))$ be the graded basis of $\fg M(a)$ defined by the $H$-frame $(X_1,\ldots, X_n)$. Then it defines a global system of coordinates that identifies the tangent group $GM(a)$ with the nilpotent Lie group $G(a)$. 
\end{proposition}

For $j=1,\ldots,n$, let $X_j^a$ be the left-invariant vector field on $G(a)$ that agrees with $\partial_j$ at $x=0$. Then $(X_1^a,\ldots, X_n^a)$ is the canonical basis of the Lie algebra of left-invariant vector fields on $G(a)$. If we let $(\epsilon_1,\ldots, \epsilon_n)$ be the canonical basis of $\R^n$, then, for $j=1,\ldots, n$ and for all $f \in C^\infty(\R^n)$,  we have
\begin{equation*}
 X_j^af(x)=\frac{d}{dt}\left.f\left(x\cdot (t\epsilon_j)\right)\right|_{t=0}= \sum_{1\leq k \leq n} \frac{d}{dt}\left.\left(x\cdot (t\epsilon_j)\right)_k\right|_{t=0} \partial_kf(x). 
\end{equation*}
Therefore, using~(\ref{eq:GM.group-law2}) we obtain
\begin{equation}
 X_j^a(x)= \partial_j +  \frac{1}{2} \sum_{w_i+w_j=w_k} L_{ij}^k(a)x_i \partial_k + \sum_{\substack{\brak\alpha+w_j=w_{k}\\ |\alpha|\geq 2}}
    B_{\alpha \epsilon_j}^k\left(L(a)\right)x^{\alpha}\partial_k,
    \label{eq:GM.Xj}
\end{equation}
where we regard $\epsilon_j$ as an element of $\N_0^n$. 

The following observation is the crux of this section. 

\begin{proposition}\label{prop:G(a)-sNX(a)}
The nilpotent Lie group $G(a)$ belongs to the class $\sN_X(a)$. 
\end{proposition}
\begin{proof}
It is immediate from~(\ref{eq:GM.group-law}) that the dilations~(\ref{eq:Nilpotent.dilations2}) are group automorphisms of $G(a)$. 
As the structure constants of the basis $(\xi_{1}(a),\ldots,\xi_{n}(a))$ are precisely given by the coefficients $L_{ij}^k(a)$, $w_i+w_j=w_k$, we see that $X_1^a,\ldots, X_n^a$ satisfy the commutator relations~(\ref{eq:Nilpotent.structure-constants-Xj(a)}). As $X_j^a(0)=\partial_j$ for $j=1,\ldots, n$, it follows that the Lie bracket of the Lie algebra $T[G(a)](0)$ is given by~(\ref{eq:Nilpotent.structure-constants-fg(a)}), and so this Lie algebra is just the Lie algebra $\fg(a)$. This shows that
 $G(a)$ is in the class $\sN_X(a)$. The proof is complete. 
\end{proof}

Proposition~\ref{prop:G(a)-sNX(a)} asserts that, under the identification given by Proposition~\ref{prop:Carnot-coord.GM(a)-G(a)}, the tangent group $GM(a)$ can be regarded as an element of the class $\sN_X(a)$. Therefore, by Corollary~\ref{cor:many.nilp-sNX(a)} it provides us with the nilpotent approximation of $(M,H)$ at $a$ in suitable privileged coordinates. The rest of the paper is devoted to the explicit construction and properties of such coordinates. 

\subsection{Carnot coordinates}
We define Carnot coordinates as follows. 

\begin{definition}
 Let $(x_{1},\ldots,x_{n})$ be privileged coordinates at $a$ adapted to $(X_{1},\ldots,X_{n})$.  We shall call $(x_{1},\ldots,x_{n})$ 
   \emph{Carnot coordinates} when, for every $j=1,\ldots, n$, the model vector of $X_j$ in these coordinates is the vector field $X_j^a$ given by~(\ref{eq:GM.Xj}). 
\end{definition}

As mentioned in Section~\ref{sec:Nilpotent-approximation}, the groups in the class $\sN_X(a)$ are uniquely determined by the canonical bases of their Lie algebras of left-invariant vector fields. Here $(X_1^a, \ldots, X_n^a)$ is the canonical basis of left-invariant vector fields on $G(a)$. Therefore, by using Corollary~\ref{cor:many.nilp-sNX(a)} we obtain the following result. 

\begin{proposition}
Let $(x_{1},\ldots,x_{n})$ be privileged coordinates at $a$ adapted to $(X_{1},\ldots,X_{n})$. Then $(x_1,\ldots, x_n)$ are Carnot coordinates if and only if in these coordinates the nilpotent approximation of $(M,H)$ at $a$ is given by the graded nilpotent graded group $G(a)$.  
\end{proposition}

In other words, the Carnot coordinates are precisely the privileged coordinates at $a$ adapted to $(X_{1},\ldots,X_{n})$ in which we have a natural identification between the nilpotent approximation and the tangent group $GM(a)$. 

In addition, the group $G(a)$ depends only on the structure constants $L_{ij}^k$, $w_i+w_j=w_k$ (\emph{cf.}~Remark~\ref{rmk:Carnot-coord.independence-G(a)}). Therefore, the Carnot coordinates provides us with a class of systems of privileged coordinates in which the nilpotent approximation is independent of the choice of the coordinate system in that class.  

Using Proposition~\ref{prop-pri-equiv} we obtain the following  characterization of Carnot coordinates. 

\begin{proposition}\label{prop:Carnot-criterion}
    Let $(x_{1},\ldots,x_{n})$ be local coordinates that are linearly adapted at $a$ to the $H$-frame 
    $(X_{1},\ldots,X_{n})$. Let $U\subset \R^n$ be their range. Then $(x_{1},\ldots,x_{n})$  are Carnot coordinates at $a$ if and only if, for $j=1,\ldots, n$  and as $t\rightarrow 0$, we have
               \begin{equation}
            t^{w_{j}}\delta_{t}^{*}X_{j}=X^{a}_{j}+\op{O}(t) \qquad \text{in $\cX(U)$}.
             \label{eq:Carnot.Xj(a)=Xja}
        \end{equation}    
\end{proposition}

Let us now explain how to construct Carnot coordinates. 

\begin{lemma}\label{lem:Carnot-coord.admissibility-Xa}
For all $x\in \R^n$, we have 
 \begin{equation}
\exp\left(x_1X_1^a+\cdots +x_nX_n^a\right)=x.   
\label{eq:Carnot-coord.expXa}
\end{equation}
\end{lemma}
\begin{proof}
The equality~(\ref{eq:Carnot-coord.expXa}) is a special case of a general result for Lie groups equipped with a Dynkin product~(\ref{eq:Carnot.Dynkin-product}). Let $\exp_{X^a}:\R^n\rightarrow \R^n$ be the diffeomorphism defined by 
\begin{equation}
\exp_{X^a}(x)=\exp\left(x_1X_1^a+\cdots +x_nX_n^a\right), \qquad x=(x_1,\ldots, x_n)\in \R^n. 
\label{eq:Carnot-coord.expXa-def}
\end{equation}
Then proving~(\ref{eq:Carnot-coord.expXa}) amounts to show that $\exp_{X^a}$ is the identity map of $\R^n$. Given  $x=(x_1,\ldots, x_n)$ and $y=(y_1,\ldots, y_n)$ in $\R^n$, set $X=x_1X_1^a+\cdots +x_nX_n^a$ and $Y=y_1X_1^a+\cdots +y_nX_n^a$. As the product~(\ref{eq:GM.group-law}) is a Dynkin product we have $\exp_{X^a}(x\cdot y)=\exp(X\cdot Y)$. Combining this with the Baker-Campbell-Hausdorff formula~(\ref{eq:Carnot.BCH-Formula}) gives
\begin{equation*}
 \exp_{X^a}(x)\cdot \exp_{X^a}(y)= \exp(X)\cdot \exp(Y)=\exp(X\cdot Y)=\exp_{X^a}(x\cdot y).
\end{equation*}
 This shows that $\exp_{X^a}$ is a group automorphism with respect to the group law~(\ref{eq:GM.group-law}). In view of the functoriality of the exponential map we then get
\begin{equation}
\exp_{X^a}\circ\exp_{X^a}(x)= \exp_{X^a}\left( \exp(X)\right)=\exp\left( (\exp_{X^a})_*X\right). 
\label{eq:Carnot-coord.expX(a)expX(a)}
\end{equation}

By definition $X_j^a(0)=\partial_j$ for $j=1,\ldots, n$. Combining this with the definition~\ref{eq:Carnot-coord.expXa-def} of $\exp_{X^a}$ we see that
\begin{equation*}
 \left(\exp_{X^a}\right)'(0)\partial_j= \partial_j\exp_{X^a}(0)= X_j^a(0)=\partial_j. 
\end{equation*}
It then follows that $(\exp_{X^a})'(0)=\op{id}$. As $\exp_{X^a}$ is a group automorphism, we see that $(\exp_{X^a})_*X$ is a left-invariant vector field which agrees at $x=0$ with $(\exp_{X^a})_*X(0)=(\exp_{X^a})'(0)(X(0))=X(0)$, and hence it agrees with $X$ everywhere. 
Combining this with~(\ref{eq:Carnot-coord.expX(a)expX(a)}) shows that $\exp_{X^a}\circ\exp_{X^a}(x)=\exp(X)=\exp_{X^a}(x)$, and hence $\exp_{X^a}(x)=x$. This proves~(\ref{eq:Carnot-coord.expXa}) and completes the proof. 
\end{proof}

The following result shows how to convert any system of privileged coordinates into a system of Carnot coordinates. 

\begin{theorem}\label{thm:Carnot-coord.w-hom-privileged-to-Carnot}
 Let $(x_{1},\ldots,x_{n})$ be privileged coordinates at $a$ adapted to the $H$-frame 
$(X_{1},\ldots,X_{n})$. For $j=1,\ldots, n$, let $X_j^{(a)}$ be the model vector fields of $X_j$ in these privileged coordinates. In addition, set $\phi=\exp_{X^{(a)}}^{-1}$, where $\exp_{X^{(a)}}$ is the diffeomorphism of $\R^n$ given by~(\ref{eq:many.exp_X(a)}). 
\begin{enumerate}
\item  The map $\phi(x)$ is the unique $w$-homogeneous polynomial diffeomorphism of $\R^n$ such that $\phi_*X_j^{(a)}=X_j^a$ for $j=1,\ldots, n$. 

\item The coordinate change $x\rightarrow \phi(x)$ is the unique $w$-homogeneous change of coordinates that gives rise to Carnot coordinates at $a$ adapted to $(X_{1},\ldots,X_{n})$. 
\end{enumerate}
\end{theorem}
\begin{proof}
We know by Lemma~\ref{lem:Carnot-coord.admissibility-Xa} that $\exp_{X^a}=\op{id}$. Thus, in the notation of Proposition~\ref{prop:Nilpotent.coord-G(a)G} we have $\phi_{X^a}=\exp_{X^{(a)}}^{-1}=\phi$. Therefore, the result follows by applying Proposition~\ref{prop:Nilpotent.coord-G(a)G} to $G=G(a)$ and the canonical basis $(X^a_1, \ldots, X^a_n)$. 
 \end{proof}

Proposition~\ref{prop:char-priv.phi-hatphi-Ow} shows how to get all systems of privileged coordinates at a given point. We have the following version of this result for Carnot coordinates. 

\begin{proposition}\label{prop:char-Carnot-coord.phi-x-Ow}
 Let $(x_1,\ldots, x_n)$ be Carnot coordinates at $a$ adapted to the $H$-frame $(X_1,\ldots, X_n)$. Then a change of coordinates $x\rightarrow \phi(x)$ produces Carnot coordinates at $a$ adapted to $(X_1,\ldots, X_n)$ if and only if we have 
\begin{equation}
 \phi(x)=x +\Ow\left(\|x\|^{w+1}\right)  \qquad \text{near $x=0$}. 
 \label{eq:char-Carnot-coord.phi-x-Ow}
\end{equation}
\end{proposition}
\begin{proof}
As mentioned in the proof of Theorem~\ref{thm:Carnot-coord.w-hom-privileged-to-Carnot}, in the notation of Proposition~\ref{prop:Nilpotent.coord-G(a)G} the map $\phi_{X^a}$ is just $\exp_{X^{(a)}}^{-1}$. 
Furthermore, as $(x_1,\ldots, x_n)$ are Carnot coordinates we have $X_j^{(a)}=X_j^a$ for $j=1,\ldots, n$, and so by Lemma~\ref{lem:Carnot-coord.admissibility-Xa} we have $\phi_{X^a}=\exp_{X^{a}}^{-1}=\op{id}$. 
Therefore, by using Proposition~\ref{prop:Nilpotent.coord-G(a)G} we see that a coordinate change $x\rightarrow \phi(x)$ yields Carnot coordinates at $a$ adapted to $(X_1,\ldots, X_n)$ if and only if $\phi(x)=\phi_{X^a}(x) +\Ow(\|x\|^{w+1})=x +\Ow(\|x\|^{w+1})$ near $x=0$. The result is  proved. 
\end{proof}

In what follows, we let $(x_{1},\ldots,x_{n})$ be privileged coordinates at $a$ adapted to the $H$-frame 
$(X_{1},\ldots,X_{n})$. Let $U$ be the range of these coordinates. We know by~(\ref{eq:Nilpotent.Xj-linearly-adapted}) that in these coordinates $X_j=\partial_{x_j} + \sum_{k=1}^n b_{jk}(x) \partial_{x_k}$, with $b_{jk}(x)\in C^\infty(U)$ such that $b_{jk}(0)=0$. 
The diffeomorphism $\phi(x)=\exp_{X^{(a)}}^{-1}(x)$ from Theorem~\ref{thm:Carnot-coord.w-hom-privileged-to-Carnot} is then determined effectively from the coefficients $b_{jk}(x)$. More precisely, we have the following result. 

\begin{proposition}\label{prop:Carnot-coord.form-phi}
Set $\phi(x)= (\phi_1(x),\ldots, \phi_n(x))$. Then, for $k=1,\ldots, n$, we have
 \begin{equation}
   \phi_{k}(x)=x_{k} -\frac{1}{4}\sum_{w_i+w_j=w_k} (\partial_i b_{jk}(0)+\partial_j b_{ik}(0))x_ix_j +\sum_{\substack{\brak\alpha=w_{k}\\|\alpha|\geq 
    3}}c_{k\alpha}x^{\alpha}, 
     \label{eq:Carnot.eps-cjalpha}
 \end{equation}where $c_{k\alpha}$ is a universal polynomial in the partial derivatives $\partial^{\beta}b_{jl}(0)$ with  
 $w_{j}+\brak\beta=w_{l}\leq w_{k}$ and $\beta\neq 0$ (\emph{cf}.\ Eq.~(\ref{eq:Carnot-coord.ckalpha}) \emph{infra}). The quadratic (resp., super-quadratic) terms only appear when $w_k\geq 2$ (resp., $w_k\geq 3$). 
 \end{proposition}
 \begin{proof}
 By the first part of Theorem~\ref{thm:Carnot-coord.w-hom-privileged-to-Carnot}, the map $\phi(x)$ is a $w$-homogeneous polynomial map such that $\phi'(0)=\op{id}$. Thus, its components $\phi_k(x)$, $k=1,\ldots, n$, must be of the form, 
\begin{equation}
 \phi_{k}(x)=x_{k}+\sum_{\substack{\brak\alpha=w_{k}\\|\alpha|\geq 
    2}}c_{k\alpha}x^{\alpha}, \qquad c_{k\alpha}\in \R.
    \label{eq:Carnot-coord.form-phik}
\end{equation}
 In particular, we see that $\phi_k(x)=x_k$ when $w_k=1$, and there is no super-quadratic terms when $w_k=2$. Therefore, the bulk of the proof is to relate the coefficients $c_{k\alpha}$ to the partial derivatives $\partial^{\beta}b_{jl}(0)$. Recall that thanks to~(\ref{eq-homo-app}) these partial derivatives allows us to compute the model vector fields $X_j^{(a)}$, $j=1,\ldots, n$. 
 
 Let $x\in \R^n$ and set $X=\phi_{1}(x)X_{1}^{(a)}+\cdots + \phi_{n}(x)X_{n}^{(a)}$. As $x=\exp_{X^{(a)}}\left(\phi(x)\right)= \exp(X)$, using 
Lemma~\ref{lem:nilpotent.flowX} we obtain
\begin{equation*}
x_k= \phi_k(x)+  
\sum_{\substack{\brak\alpha=w_{k}\\ |\alpha|\geq 2}} \hat{c}_{k\alpha} \phi(x)^\alpha, \qquad k=1,\ldots, n, 
\end{equation*}
where $\hat{c}_{k\alpha}$ is  a universal polynomial in the partial derivatives $\partial^\beta b_{jl}(0)$ with  $w_{j}+\brak\beta=w_{l}\leq w_{k}$ and $w_j<w_l$. 
We also observe that if $\brak\alpha=w_{k}$ and $|\alpha|\geq 2$, then $\phi(x)^\alpha$ only involve components $\phi_l(x)$ with $w_l<w_k$. Thus, 
\begin{equation}
 \phi_k(x)=x_k - \sum_{\substack{\brak\alpha=w_{k}\\ |\alpha|\geq 2}} \hat{c}_{k\alpha} \prod_{w_l<w_k} \biggl( x_l + \sum_{\substack{\brak \beta=w_l\\ |\beta|\geq 2}} 
 c_{l\beta}x^\beta\biggr)^{\alpha_l}. 
 \label{eq:Carnot-coord.ckalpha}
\end{equation}
Comparing this to~(\ref{eq:Carnot-coord.form-phik}) enables us to express the coefficients $c_{k\alpha}$ as polynomials in the coefficients $\hat{c}_{k\gamma}$ with $\brak \gamma =w_k$ and $\gamma\neq 0$ and the coefficients $c_{l\beta}$ with $\brak\beta=w_l<w_k$ and $\beta\neq 0$. 
A simple induction on $w_k$ then shows that each coefficient $c_{k\alpha}$ is a universal polynomial in the partial derivatives $\partial^\beta b_{jl}(0)$ with $w_j+\brak\beta=w_l\leq w_k$ and $w_j<w_k$. 

It remains to compute the quadratic part of $\phi_k(x)$ when $w_k\geq 2$. Using~(\ref{eq:Carnot-coord.ckalpha}) we get
\begin{equation}
 \phi_k(x) =x_k - \sum_{\substack{\brak\alpha=w_{k}\\ |\alpha|= 2}} \hat{c}_{k\alpha}x^\alpha +R_k(x),
 \label{eq:Carnot-coord.phik-quadratic}
\end{equation}
where $R_k(x)$ is a super-quadratic polynomial. Therefore, in order to determine the quadratic part of the $\phi_k(x)$ we only have to compute 
 the coefficients $\hat{c}_{k\alpha}x^\alpha$ with $\brak\alpha=w_k$ and $|\alpha|=2$. \\
 
Let $\xi \in \R^n$. For $t\in \R$ set $x(t)=\exp(tX)(0)$, with $X=\xi_1X_1^{(a)}+\cdots +  \xi_nX_n^{(a)}$. It follows from Lemma~\ref{lem:nilpotent.flowX} that we have
 \begin{equation}
x_k(t)=t\xi_k + \sum_{\substack{\brak\alpha=w_k\\ |\alpha|=2}} \hat{c}_{k\alpha}\xi^\alpha t^2 +y_k(t),
\label{eq:Carnot.flowXk-yk}
\end{equation}
where $y_k(t)$ is a super-quadratic polynomial in $t$. Differentiating this with respect to $t$ gives
\begin{equation}
\dot{x}_k(t)=\xi_k + 2\sum_{\substack{\brak\alpha=w_k\\ |\alpha|=2}} \hat{c}_{k\alpha}\xi^\alpha t +\dot{y}_k(t). 
\label{eq:Carnot-coord.dotxk-affine}
\end{equation}
Moreover, using~(\ref{eq:Carnot-coord.ODE-system}) and~(\ref{eq:Carnot.flowXk-yk}) we obtain 
\begin{equation*}
\dot{x}_k(t)= \xi_{k}+ \sum_{\substack{w_{j}+\brak\alpha=w_{k}\\
     w_k>w_j}} \xi_{k} b_{jk\alpha} \prod_{w_l<w_k} x_l(t)^{\alpha_l}= 
\xi_k +  \sum_{\substack{w_{j}+\brak\alpha=w_{k}\\
     |\alpha|=1}} b_{jk\alpha} \xi^{\alpha} \xi_{j}t+ z_k(t),
\end{equation*}
where  we have set $b_{jk\alpha}=(\alpha !)^{-1}\partial^\alpha b_{jk}(0)$ and $z_k(t)$ is a super-linear polynomial in $t$. 
Comparing this to~(\ref{eq:Carnot-coord.dotxk-affine}) we see that 
\begin{equation*}
\sum_{\substack{\brak\alpha=w_k\\ |\alpha|=2}} \hat{c}_{k\alpha}\xi^\alpha= \frac{1}{2}\! \sum_{\substack{w_{j}+\brak\alpha=w_{k}\\
     |\alpha|=1}}  b_{jk\alpha} \xi^{\alpha}\xi_{j}=\frac{1}{4}\! \sum_{w_i+w_j=w_k} \left( \partial_ib_{jk}(0)+ \partial_jb_{ik}(0)\right) \xi_i\xi_j. 
\end{equation*}
Substituting $x$ for $\xi$ and using~(\ref{eq:Carnot-coord.phik-quadratic}) then shows that the quadratic part of $\phi_k(x)$ is exactly as given in~(\ref{eq:Carnot.eps-cjalpha}). The proof is complete. 
 \end{proof}

\subsection{Examples of Carnot coordinates} \label{subsec:examples-Carnot-coord}
Let us now look at the construction of Carnot coordinates in some examples. We also refer to  Section~\ref{sec:normal-group} for the description of $\varepsilon$-Carnot coordinates on graded nilpotent Lie groups.

\subsubsection*{A. The Step 1 case} Suppose that $r=1$ so that $H_1=H_r=TM$. Let $X_1, \ldots, X_n$ be a tangent frame near a point $a\in M$. In this case, as a Lie algebra bundle $\fg M=TM$ and as a group bundle $GM=TM$. For $j=1,\ldots, n$, set $\xi_j=X_j(a)$. Then $(\xi_1,\ldots, \xi_n)$ is a basis of $\fg M(a)$ and with respect to the coordinates defined by this basis $\fg M(a)$ and $GM(a)$ are identified with $\R^n$ as Lie algebra and Lie group, respectively. In particular, for $j=1,\ldots, n$, the left-invariant vector fields $X_j^a$ is just the partial derivative $\partial_j$. 

Given local coordinates near $a$, set $X_j=\sum b_{jk}(x)\partial_k$ and $B_X(a)=(b_{jk}(a))_{1\leq j,k\leq n}$. The affine map $T_ax=(B_X(a)^t)^{-1}(x-a)$ provides us with linearly adapted coordinates at $a$, so that in these coordinates $X_j(0)=\partial_j$ for $j=1,\ldots, n$. The dilations~(\ref{eq:Nilpotent.dilations2}) are given by the usual scalar multiplication, and so in these coordinates $X_j^{(a)}=X_j(0)=\partial_j=X_j^a$. Therefore, these coordinates are Carnot coordinates at $a$. 

\subsubsection*{B. Some step 2 examples} Suppose now that $(M^{d+1}, H)$ is a Heisenberg manifold in the sense of~\cite{BG:CHM}. In this case $H\subset TM$ is a hyperplane bundle, and the Carnot manifold structure is defined by the Carnot filtration $(H, TM)$, which has type $(d,d+1)$. The Lie group structure on the fibers of $GM$ is encoded by the Levi form,
 \begin{equation}
 \cL: H\times H \longrightarrow TM/H,
 \label{eq:tangent-group.Levi-form} 
\end{equation}
which is the bilinear map $\cL_{w_1,w_2}$ of Lemma~\ref{lem-dep} for $w_1=w_2=1$. Given $a\in M$, the group law on $GM(a)=H(a)\oplus (TM(a)/H(a))$ is such that, for all $\xi, \eta \in H(a)$ and $\xi_0,\eta_0 \in TM(a)/H(a)$, we have 
 \begin{equation}
( \xi+\xi_0)\cdot (\eta+\eta_0)= \xi+\eta+ \xi_0 +\eta_0 + \frac{1}{2} \cL(\xi,\eta). 
\label{eq:GM.Heisenberg-mfld-group-law}
\end{equation}

Let $(X_1,\ldots, X_d, T)$ be an $H$-frame near $x=a$,  so that $(X_1,....,X_d)$ is a local frame of $H$ and $T$ is a transverse vector field. Near $x=a$, there are smooth functions $L_{jk}(x)$, $j,k=1,\ldots, d$, such that
\begin{equation*}
[X_j,X_k](x)=L_{jk}(x)T(x) \qquad \bmod H.  
\end{equation*}
For $j=1,\ldots, d$ set $\xi_j=X_j(a)$. In addition, let $\tau$ be the class of $T(a)$ in $TM(a)/H(a)$. Then the Levi form~(\ref{eq:tangent-group.Levi-form}) is such that
\begin{equation}
\cL(\xi_i,\xi_j)=L_{ij}(a)\tau, \qquad  \text{for $i,j=1,\ldots, d$}. 
\label{eq:Carnot-coord.Levi-form-XjXk}
\end{equation}
Let us denote by $(x_1,\ldots, x_d,t)$ the coordinates on $\fg M(a)$ and $GM(a)$ defined by the basis $(\xi_1,\ldots, \xi_d, \tau)$. Combining~(\ref{eq:GM.Heisenberg-mfld-group-law}) and~(\ref{eq:Carnot-coord.Levi-form-XjXk}) shows that, in these coordinates, the group law of $GM(a)$ is given by
\begin{equation}
(x_1,\ldots, x_d,t)\cdot (x_1',\ldots, x_d',t')=\biggl(x_1+x_1', \ldots, x_d+x_d', t+t'+\frac{1}{2}L(a)(x,x')\biggr), 
\label{eq:GM.Heisenberg-mfld.group-law}
\end{equation}
where we have set $L(a)(x,x')=\sum L_{ij}(a)x_ix_j$. 
In particular, the left-invariant vector fields $T^a$ and $X_1^a,\ldots, X_d^a$ are given by 
\begin{equation}
T^a=\partial_t, \qquad X_j^a= \partial_j +\frac12  \sum_{1\leq i \leq d} L_{ij}(a)x_i\partial_t, \quad 1\leq j \leq d.
 \label{eq:Carnot-coord.Xja-Heisenberg-mfld}
\end{equation}

Bearing this in mind, starting with local coordinates near $a$, the affine change of variable $x\rightarrow T_ax$ provides us with local coordinates $(x_1,\ldots, x_d,t)$ that are linearly adapted at $a$. Thus, in these coordinates $X_j=\sum b_{jk}(x,t)\partial_{x_k}+ c_j(x,t)\partial_t$, $j=1,\ldots,d$, where the coefficients $b_{jk}(x,t)$ and $c_j(x,t)$ are smooth functions of the variables $(x_1,\ldots, x_d,t)$ near the origin such that $b_{jk}(0,0)=\delta_{jk}$ and $\tilde{b}_j(0,0)=0$. We then have 
\begin{equation}
T^{(a)}=\partial_t, \qquad X_j^{(a)}= \partial_{x_j} + \sum_{1\leq i \leq d} c_{ji}x_i\partial_t, \quad 1\leq j \leq d, 
\label{eq:Carnot-coord.Xj(a)-Heisenberg-mfld}
\end{equation}
where we have set $c_{ji}= \partial_{x_i}c_j(0,0)$.  Here $T^{(a)}$ is homogeneous of degree $-2$ and $X_1^{(a)}, \ldots, X_d^{(a)}$ are homogeneous of degree~$-1$. Combining this with Proposition~\ref{prop-pri-equiv} shows that we are dealing with privileged coordinates. Thus, in this case $\psi_a=\op{id}$.
Moreover, as $r=2$ the components of $\phi(x):=\exp_{X^{(a)}}^{-1}(x)$ only involve linear and quadratic terms. In fact, it follows from~(\ref{eq:Carnot.eps-cjalpha}) and~(\ref{eq:Carnot-coord.Xj(a)-Heisenberg-mfld}) that we have
\begin{equation}
\phi(x_1,\ldots,x_d,t)= \biggl( x_1,\ldots, x_d, t- \frac{1}{4} \sum_{1\leq i,j\leq d} (c_{ij}+c_{ji})x_ix_j\biggr). 
\label{eq:Carnot-coord.phi-r2}
\end{equation}
All this shows that, with $\phi(x)$ and $T_ax$ as above, the change of coordinates $x\rightarrow \phi(T_ax)$ provides us with Carnot coordinates at $a$ adapted to 
$(X_1,\ldots, X_n)$. 

Assume further that $n=2m+1$ and $(M,H)$ is a contact manifold. Near $a$ we thus have a contact $1$-form $\theta$ such that $\ker \theta =H$. Let $(x_1,\ldots, x_{2m},t)$ be Darboux coordinates centered at $a$ such that
\begin{equation*}
\theta = dt + \frac12 \sum_{1\leq j \leq m} (x_{m+j}\partial_j - x_j\partial_{m+j}). 
\end{equation*}
This is the standard left-invariant contact form on the Heisenberg group $\bH^{2m+1}$, i.e., $\R^{2m+1}$ equipped with the group law~(\ref{eq:GM.Heisenberg-mfld.group-law}) associated with the antisymmetric bilinear form,
\begin{equation*}
 L(x,x')= \sum_{j=1}^m  (x_{j}x'_{m+j}-x_{m+j}x_{j}'). 
\end{equation*}
The canonical basis of the Lie algebra of left-invariant vector fields  on $\bH^{2m+1}$ is given by
    \begin{equation}
      T=\partial_{t}, \quad X_{j}=\partial_{x_{j}}-\frac{1}{2}x_{m+j}\partial_{t}, \quad X_{m+j}=\partial_{x_{m+j}}+\frac{1}{2}x_{j} 
      \partial_{t},  \quad j=1,\ldots,m.  
        \label{eq:many.Heisenberg-Lie-algebra-basis}  
    \end{equation}
We see from~(\ref{eq:Carnot-coord.Xja-Heisenberg-mfld}) that $T^a=T$ and $X_j^a=X_j$, $j=1,\ldots, 2m$, so that $G(a)=\bH^{2m+1}$. Moreover, as $T$ is homogeneous of degree~$-2$ and $X_1, \ldots, X_{2m}$ are homogeneous of degree~$-1$, we also see that $T^{(a)}=T=T^a$ and $X_j^{(a)}=X_j=X_j^a$. This shows that $(x_1, \ldots,x_m,t)$ are Carnot coordinates at $a$. Therefore, we arrive at the following result. 

\begin{proposition}\label{prop:Carnot-coor.Darboux-coord}
 Suppose that $(M,H)$ is a contact manifold. Given any point $a\in M$, all Darboux coordinates centered at $a$ are Carnot coordinates with respect to the $H$-frame~(\ref{eq:many.Heisenberg-Lie-algebra-basis}). 
\end{proposition}

Further examples of Carnot coordinates are described in Section~\ref{sec:Canonical-coordinate} and Section~\ref{sec:vareps-Carnot-coord}. 

\section{Canonical coordinates}\label{sec:Canonical-coordinate}  
In this section, we show that the canonical coordinates of the 1st kind of~\cite{Go:LNM76, RS:ActaMath76} are Carnot coordinates. We also show that, although they are privileged coordinates, the canonical coordinates of the 2nd kind of~\cite{BS:SIAMJCO90, He:SIAMR91} are not Carnot coordinates. 

Throughout this section we let $(X_1,\ldots, X_n)$ be an $H$-frame near a given point $a\in M$. There is an open neighborhood $V$ of the origin in $\R^n$ such that, for all $x=(x_1, \ldots, x_n)$ in $V$ the flow $\exp[t(x_1X_1+\cdots +x_nX_n)](a)$ exists for all times $t\in [-1,1]$ and depends smoothly on $(x_1,\ldots, x_n)$ and $t$. We thus have a smooth map $V\ni x\rightarrow \exp_X(x;a)\in M$, where 
\begin{align}
\exp_{X}(x;a)& =\exp(x_1X_1+\cdots +x_nX_n)(a)\nonumber \\
& = \left.\exp[t(x_1X_1+\cdots +x_nX_n)](a)\right|_{t=1}. 
\label{eq:Can-coord.1st-kind}
\end{align}
For $j=1,\ldots, n$, we have $\partial_{x_j} \exp_{X}(0;a)= X_j(a)$. As $(X_1(a),\ldots, X_n(a))$ is a basis of $TM(a)$, we deduce that $\exp_X'(0;a)$ is non-singular. Therefore, possibly by shrinking $V$, we may assume that $x\rightarrow \exp_X(x;a)$ is a diffeomorphism from $V$ onto an open neighborhood of $a$ in $M$. Its inverse map then is a local chart around $a$. The local coordinates defined by this chart are the canonical coordinates of the first kind of Goodman~\cite{Go:LNM76} and Rothschild-Stein~\cite{RS:ActaMath76}. 

We quote the following result from~\cite{CP:Privileged}. 
 
 \begin{lemma}[\cite{CP:Privileged}]\label{lem:can.1stkind-approx}
 Suppose that $(x_1,\ldots, x_n)$ are privileged coordinates at $a$ adapted to $(X_1,\ldots, X_n)$. Then, in these coordinates and near $x=0$, we have
\begin{align*}
 \exp(x_1X_1+\cdots +x_nX_n)(0)= \exp\left(x_1X_1^{(a)}+\cdots + x_nX_n^{(a)}\right)\!(0)+\Ow\left(\|x\|^{w+1}\right).
\end{align*}
\end{lemma}

We are now in a position to prove the following result. 

\begin{proposition}\label{prop:Canonical-coord.1st-kind}
 The canonical coordinates of the 1st kind defined by~(\ref{eq:Can-coord.1st-kind}) are Carnot coordinates at $a$ adapted to the $H$-frame $(X_1,\ldots, X_n)$. 
\end{proposition}
\begin{proof}
 The local coordinates at stake are defined by the local chart  $\kappa_a$ that inverts the map $V\ni x\rightarrow \exp_X(x;a)$. Let $\kappa$ be a local chart that gives rise to Carnot coordinates at $a$ adapted to $(X_1, \ldots, X_n)$. Without any loss of generality we may assume that $\kappa_a$ and $\kappa$ have the same domain. Set $\phi=\kappa_a\circ \kappa^{-1}$. If we denote by $(x_1, \ldots, x_n)$ the local coordinates defined by $\kappa$, then we pass from these coordinates to the canonical coordinates of the 1st kind by means of the change of coordinates $x\rightarrow \phi(x)$. Therefore, by Proposition~\ref{lem:Carnot-coord.inverse-Ow} the coordinates provided by $\kappa_a$ are Carnot coordinates at $a$ adapted to $(X_1,\ldots, X_n)$  if and only if $\phi(x)=x+\Ow(\|x\|^{w+1})$ near $x=0$. Combining this with Proposition~\ref{lem:Carnot-coord.inverse-Ow} further shows that we have the desired result if and only $\phi^{-1}(x)=x+\Ow(\|x\|^{w+1})$ near $x=0$.

 Bearing all this in mind, we observe that, for all $x\in V$, we have 
\begin{align*}
 \phi^{-1}(x) & = \kappa\left[ \exp(x_1X_1+\cdots +x_nX_n)(a)\right] \\
 & = \kappa \circ  \exp\left( x_1\kappa_*X_1+\cdots + x_n\kappa_*X_n\right)\circ \kappa^{-1}(0)\\
 & = \exp\left( x_1\kappa_*X_1+\cdots + x_n\kappa_*X_n\right)(0). 
\end{align*}
 In other word, the map $\phi^{-1}$ is the map $x\rightarrow \exp(x_1X_1+\cdots +x_nX_n)(0)$ in the Carnot coordinates defined by $\kappa$. In these coordinates, the model vector field of each vector field $X_j$, $j=1,\ldots, n$,  is $X_j^a$. Therefore, by using Lemma~\ref{lem:Carnot-coord.admissibility-Xa} and Lemma~\ref{lem:can.1stkind-approx} we see that, near $x=0$, we have 
\begin{equation}
 \phi^{-1}(x)=\exp_{X^{a}}(x) +\Ow\left(\|x\|^{w+1}\right) = x +\Ow\left(\|x\|^{w+1}\right). 
  \label{eq:Can.phi-1-expX}
\end{equation}
This proves the result. 
\end{proof}

Let us now look at the canonical coordinates of the 2nd kind of~\cite{BS:SIAMJCO90, He:SIAMR91}. There is an open neighborhood $W$ of the origin $0\in \R^n$ such that, for all $(x_1, \ldots, x_n)\in W$, the composition of flows $\exp(t x_1X_1)\circ \cdots \circ \exp(tx_nX_n)(a)$ is well defined for all $t\in [-1,1]$ and depends smoothly on $(x_1,\ldots, x_n)$ and $t$. Thus, we obtain a smooth map $W\ni x \rightarrow \gamma_X(x;a)\in M$, where 
\begin{align}
  \gamma_X(x;a)& =  \exp\left( x_1X_1\right) \circ \cdots \circ  \exp\left(x_nX_n\right)\!(a)\nonumber \\
  & = \left.  \exp\left( tx_1X_1\right) \circ \cdots \circ  \exp\left(x_nX_n\right)\!(a)\right|_{t=1}. 
  \label{eq:Carnot-coord.can-coord2nd}
\end{align}
For $j=1,\ldots, n$, we have $\partial_{x_j} \gamma_X(0)= \partial_{x_j} \exp(x_jX_j)(a)= X_j(a)$. Therefore, in the same way as with the map $x\rightarrow \exp_X(x; a)$ above, possibly by shrinking $W$, we may assume that $x\rightarrow  \gamma_X(x;a)$ is a diffeomorphism from $W$ onto its range. Its inverse map then defines a local chart around $a$. The local coordinates defined by that chart are the canonical coordinates of the 2nd kind of Bianchini-Stefani~\cite{BS:SIAMJCO90} and Hermes~\cite{He:SIAMR91}.

\begin{lemma}[\cite{CP:Privileged}]\label{prop:can.2ndkind-approx}
 Suppose that $(x_1,\ldots,x_n)$ are privileged coordinates at $a$ adapted to $(X_1,\ldots, X_n)$. Then, in these coordinates and near $x=0$, we have
\begin{equation*}
\exp\left( x_1X_1\right) \circ \cdots \circ  \exp\left(x_nX_n\right)\!(0)= \exp\left( x_1X_1^{(a)}\right) \circ \cdots \circ  \exp\left(x_nX_n^{(a)}\right)\!(0) +\Ow\left(\|x\|^{w+1}\right). 
\end{equation*}
\end{lemma}

As mentioned in Remark~\ref{rmk:privileged.canonical-coord}, the canonical coordinates of the 2nd kind are privileged coordinates. However, unlike the canonical coordinates of the 1st kind, in general these coordinates are not Carnot coordinates. More precisely, we have the following result. 
 
\begin{proposition}\label{prop:Canonical-coord.2nd-kind}
 If $r\geq 2$, then the canonical coordinates of the second kind given by~(\ref{eq:Carnot-coord.can-coord2nd}) are never Carnot coordinates.
\end{proposition}
\begin{proof}
We proceed as in the proof of Proposition~\ref{prop:Canonical-coord.1st-kind}.  
Let $\hat{\kappa}_a:U_0\rightarrow V$ be the inverse of the diffeomorphism~(\ref{eq:Carnot-coord.can-coord2nd}). This is the local chart that defines the canonical coordinates of the 2nd kind. Let $\kappa$ be a local chart around $a$ that gives rise to a system of Carnot coordinates at $a$ adapted to $(X_1,\ldots, X_n)$. As in the proof of Proposition~\ref{prop:Canonical-coord.1st-kind}, we may assume that $\kappa_a$ and $\kappa$ have the same domain. Set $\phi=\kappa_a\circ \kappa^{-1}$ and denote by $(x_1,\ldots, x_n)$ the local coordinates defined by $\kappa$. In the same way as in the proof of Proposition~\ref{prop:Canonical-coord.1st-kind} the canonical coordinates of the 2nd kind are Carnot coordinates are $a$ adapted to $(X_1,\ldots, X_n)$ if and only if $\phi^{-1}(x)=x+\Ow(\|x\|^{w+1})$ near $x=0$.

Bearing this in mind, by arguing as in the proof of Proposition~\ref{prop:Canonical-coord.1st-kind} and using the fact that we are working in Carnot coordinates, we see that, near $x=0$, we have 
\begin{align*}
 \phi^{-1}(x) & = \exp\left(x_1X_1^{(a)}\right)\circ \cdots \exp\left(x_nX_n^{(a)}\right)\!(0)+ \Ow\left(\|x\|^{w+1}\right) \\
 & =  \exp\left(x_1X_1^a\right)\circ \cdots \exp\left(x_nX_n^a\right)\!(0)+ \Ow\left(\|x\|^{w+1}\right). 
\end{align*}
Recall that by Lemma~\ref{lem:Carnot-coord.admissibility-Xa} we have $\exp(x_1X_1^a+\cdots + X_n^a)=x$ for all $x\in \R^n$. Therefore, we deduce that  the canonical coordinates of the 2nd kind are Carnot coordinates are $a$ adapted to $(X_1,\ldots, X_n)$ if and only if we have 
\begin{equation*}
  \exp\left(x_1X_1^a\right)\circ \cdots \exp\left(x_nX_n^a\right)\!(0) = \exp(x_1X_1^a+\cdots + X_n^a)(0) \qquad \text{for all $x\in \R^n$}. 
\end{equation*}
This happens if and only if the Lie algebra generated by $X_1^a, \ldots, X_n^a$ is Abelian, i.e., $r=1$. Therefore, if $r\geq 2$, then the canonical coordinates of the 2nd kind cannot be Carnot coordinates at $a$. The proof is complete. 
 \end{proof}

\section{$\varepsilon$-Carnot Coordinates}\label{sec:vareps-Carnot-coord}
In this section, we single out a special type of polynomial Carnot coordinates, which we call $\varepsilon$-Carnot coordinates. They are obtained by converting the privileged coordinates of Proposition~\ref{prop-uni} and~\cite{Be:Tangent} into Carnot coordinates. As we shall see, on graded nilpotent groups these coordinates arise from the group multiplication. 

Throughout this section we let $(X_1,\ldots, X_n)$ be an $H$-frame near a given point $a \in M$. 

\subsection{$\varepsilon$-Carnot coordinates} Suppose we are given local coordinates $(x_{1},\ldots,x_{n})$ near $a$.  Combining Theorem~\ref{thm:Carnot-coord.w-hom-privileged-to-Carnot} with the construction of privileged coordinates presented in Section~\ref{sec:Privileged} enables us to obtain an \emph{effective} construction of Carnot coordinates as follows. 

\begin{definition} 
For $j=1,...,n$, let $X_j^{(a)}$ be the model vector field of $X_j$ in the $\psi$-privileged coordinates provided by Proposition~\ref{prop:privileged.psi-privileged}. Then
    \begin{enumerate}
        \item  The map $\hat{\varepsilon}_{a}:\R^n \rightarrow \R^n$ is the composition $\exp_{X^{(a)}}^{-1}\circ \hat{\psi}_{a}$, where $\hat{\psi}_{a}$ is as in 
        Definition~\ref{def-psia} and $\exp_{X^{(a)}}$ is given by~(\ref{eq:many.exp_X(a)}). 
            
        \item  The map $\varepsilon_{a}:\R^n \rightarrow \R^n$ is the composition $\exp_{X^{(a)}}^{-1}\circ \psi_{a}$, where $\hat{\psi}_{a}$ is as in 
        Definition~\ref{def-psia}.   
    \end{enumerate}
\end{definition}

The $\psi$-privileged coordinates are provided by the change of coordinates $x\rightarrow \psi_a(x)$. Therefore, by using Theorem~\ref{thm:Carnot-coord.w-hom-privileged-to-Carnot} we arrive at the following result. 

\begin{proposition}\label{prop:Carnot-coord.carnot-cor}
    The change of coordinates $x\rightarrow \varepsilon_{a}(x)$ provides us with Carnot coordinates at $a$ adapted to the 
    $H$-frame $(X_{1},\ldots,X_{n})$. 
\end{proposition}

\begin{definition}
    The Carnot coordinates provided by the change of variables $x\rightarrow \varepsilon_{a}(x)$ are called \emph{$\varepsilon$-Carnot coordinates}. 
    The map $\varepsilon_{a}$ is called the \emph{$\varepsilon$-Carnot coordinate map}. 
\end{definition}

\begin{remark}
 For $j=1,\ldots, n$, let $\hat{X}_j^{(a)}$ be the pullback by $x\rightarrow \psi_a(x)$ of the model vector field $X_j^{(a)}$. Then we have
\begin{align*}
 \varepsilon_a^{-1}(x) &  = \psi_a^{-1}\circ \exp_{X^{(a)}}(x)\\
     & = \psi_a^{-1}\circ \exp\left(x_1X_1^{(a)} + \cdots + x_1X_1^{(a)}\right)\!(0)\\ 
      & = \exp\left(x_1\hat{X}_1^{(a)} + \cdots + x_1\hat{X}_1^{(a)}\right)\!(a)
\end{align*}
 Therefore, the $\varepsilon$-Carnot coordinates are the canonical coordinates of the 1st kind in the sense of~\cite{Go:LNM76, RS:ActaMath76} that are associated with the frame $(\hat{X}_1^{(a)}, \ldots,\hat{X}_n^{(a)})$. 
\end{remark}

 \begin{remark}\label{rem-phi-hat}
  The maps $\hat{\psi}_a$ and $\psi_a$ are related by $\psi_a=\hat{\psi}_a\circ T_a$, where $T_{a}$ is the affine map from Lemma~\ref{lem-affine}. Thus, 
\begin{equation*}
 \varepsilon_a = \exp_{X^{(a)}}^{-1}\circ  \psi_a = \exp_{X^{(a)}}^{-1}\circ  \hat{\psi}_a\circ T_a = \hat{\varepsilon}_a\circ T_a.  
 \end{equation*}
Moreover, it follows from Proposition~\ref{prop-uni} and Proposition~\ref{prop:Carnot-coord.form-phi} that the components $\hat{\varepsilon}_{a}(x)_k$, $k=1, \ldots, n$, are of the form,
     \begin{equation*}
        (\hat{\varepsilon}_{a})_{k}(x)=x_{k}+\sum_{\substack{\brak\alpha \leq w_{k}\\|\alpha|\geq 2}}c_{k\alpha}x^{\alpha}, 
     \end{equation*}where the coefficients $c_{k\alpha} \in \R$ are as follows:
     \begin{itemize}
         \item[-] If $\brak\alpha=w_{k}$, then they agree with the coefficients $c_{k\alpha}$ in~(\ref{eq:Carnot.eps-cjalpha}).  
     
         \item[-] If $\brak \alpha<w_{k}$, then they are universal polynomials in the coefficients $c_{l\beta}$ with $\brak\beta=w_{l}<w_{k}$
         and the coefficients $a_{l\gamma}$ in~(\ref{eq-form-h}) with  $\brak \gamma<w_{l}<w_{k}$.
    \end{itemize}
     We refer to Proposition~\ref{prop:Carnot-prop.polynomial-X} for a more precise description of the coefficients $c_{k\alpha}$. 
 \end{remark}

Combining Theorem~\ref{thm:Carnot-coord.w-hom-privileged-to-Carnot} and Proposition~\ref{prop:Carnot-coord.carnot-cor} we obtain the following description of the systems of Carnot coordinates. 

\begin{corollary}\label{cor:vareps-Carnot.all-Carnot-coord}
 The systems of Carnot coordinates at $a$ adapted to $(X_1,\ldots, X_n)$ are exactly the coordinate systems defined by local charts of the form $(\op{id}+\Theta)\circ \varepsilon_{\kappa(a)} \circ \kappa$, where $\kappa$ is a local chart near $a$ such that $\kappa(a)=0$ and $\Theta$ is a (smooth) $\Ow(\|x\|^{w+1})$-map near the origin $0\in \R^n$.  
\end{corollary}

We further have the following characterization of $\varepsilon$-Carnot coordinates. 

\begin{theorem}\label{thm:Carnot-coord.unicity-normal2}
      The $\varepsilon$-Carnot coordinates are the unique Carnot coordinates at $a$
      adapted to the $H$-frame $(X_{1},\ldots,X_{n})$ that arise from a change of coordinates of the form $x\rightarrow (\hat{\varepsilon}\circ T)(x)$,  where 
  $T(x)$ is an affine map such that $T(a)=0$, and $\hat{\varepsilon}(x)$ is a polynomial diffeomorphism whose components $\hat{\varepsilon}_{k}(x)$, $k=1,\ldots,n$, are of the form, 
                      \begin{equation}
                               \hat{\varepsilon}_{k}(x)=x_{k}+\sum_{\substack{\brak\alpha \leq w_{k}\\|\alpha|\geq 2}}d_{k\alpha}x^{\alpha}, \qquad d_{k\alpha} \in \R. 
                              \label{eq-phi-hat}
                         \end{equation}
\end{theorem}
\begin{proof}
As mentioned in Remark~\ref{rem-phi-hat} we have $ \varepsilon_a = \hat{\varepsilon}_a\circ T_a$. Here $T_a$ is an affine map such that $T_a(a)=0$ and $\hat{\varepsilon}_a(x)$ of the form~(\ref{eq-phi-hat}). 

Conversely, consider a change of coordinate of the form $x\rightarrow (\hat{\phi}\circ T)(x)$,
 where  $T(x)$ is an affine map such that $T(a)=0$ and $\hat{\phi}$ is a polynomial diffeomorphism of the form~(\ref{eq-phi-hat}).
  In the same way as in the proof of Proposition~\ref{prop-uni} it can be shown that $T=T_{a}$. Moreover, by Remark~\ref{rem-phi-hat}
 we know that $\hat{\varepsilon}_{a}(x)$ is a polynomial diffeomorphism of the 
  form~(\ref{eq-phi-hat}). Therefore, in order to complete the proof we only need to show that $\hat{\phi}=\hat{\varepsilon}_{a}$. To achieve this we will make use of the following observation.
  
\begin{claim}
Let $\cE$ be the class of diffeomorphisms of $\R^n$ of the form~(\ref{eq-phi-hat}). Then $\cE$ is a subgroup of the diffeomorphism group of $\R^n$.  
\end{claim}
\begin{proof}[Proof of the Claim] 
Let $\psi(x)$ be a diffeomorphism in the class $\cE$. For $k=1,\ldots, n$, we have
 \begin{equation}
     \psi_{k}(x)=x_{k}+\sum_{\substack{\brak\alpha \leq w_{k}\\|\alpha|\geq 2}}d_{k\alpha}x^{\alpha}, \qquad d_{k\alpha} \in \R. 
     \label{eq:Carnot-coord.classcE}
\end{equation}
In the same way as in the proof of Proposition~\ref{prop:Carnot-coord.form-phi}, we observe that if $\brak\alpha\leq w_k$ and $|\alpha|\geq 2$, then the monomial $x^\alpha$ does not involve variables $x_l$ with $w_l\geq w_k$. Therefore, if $\tilde{\psi}(x)$ is another element of $\cE$, then, for $k=1,\ldots,n$, we have
 \begin{equation*}
     \psi_{k}\circ\tilde{\psi} (x)=\tilde{\psi}_{k}(x)+\sum_{\substack{\brak\alpha \leq w_{k}\\|\alpha|\geq 2}}d_{k\alpha}\prod_{w_l<w_k}\tilde{\psi}_l(x)^{\alpha_l}. 
\end{equation*}
As the components $\tilde{\psi}_l(x)$ of $\tilde{\psi}(x)$ are of the form~(\ref{eq-phi-hat}), it then follows that $ \psi_{k}\circ\tilde{\psi} (x)$ is of the form~(\ref{eq-phi-hat}) for $k=1,\ldots,n$. That is, the diffeomorphism $ \psi\circ\tilde{\psi}(x)$ is in the class $\cE$.   

In addition, substituting $\psi^{-1}(x)$ for $x$ in~(\ref{eq:Carnot-coord.classcE}) and solving for $\psi^{-1}_k(x)$ gives 
 \begin{equation*}
     \psi_{k}^{-1}(x)=x_{k}-\sum_{\substack{\brak\alpha \leq w_{k}\\|\alpha|\geq 2}}d_{k\alpha}\prod_{w_l<w_k}\left(\psi^{-1}_j(x)\right)^{\alpha_j}, \qquad k=1,\ldots, n.  
\end{equation*}
In particular, we see that $\psi^{-1}_k(x)=x_k$ when $w_k=1$. More generally, an induction on $w_k$ shows that $\psi_{k}^{-1}(x)$ is of the form~(\ref{eq-phi-hat}) for $k=1,\ldots,n$. That is, the inverse map $\psi^{-1}(x)$  is in the class $\cE$. Therefore, we see that $\cE$ is a subgroup of the group of diffeomorphism group of $\R^n$. The claim is thus proved. 
\end{proof}
 
Let us go back to the proof of Theorem~\ref{thm:Carnot-coord.unicity-normal2}. The claim ensures us that $ \hat{\phi}\circ \hat{\varepsilon}_{a}^{-1}(x)$ is a diffeomorphism of the form~(\ref{eq-phi-hat}). In particular, each component $ \hat{\phi}_k\circ \hat{\varepsilon}_{a}^{-1}(x)-x_k$ is either zero or has weight~$\leq w_k$. Moreover, as the coordinate change $x\rightarrow \phi(x)$ provides us with Carnot coordinates, Proposition~\ref{prop:char-Carnot-coord.phi-x-Ow} ensures us that, in the Carnot coordinates provided by $\varepsilon_a$, the diffeomorphism $\phi$ has a behavior of the form~(\ref{eq:char-Carnot-coord.phi-x-Ow}) near the origin. This means that $\phi \circ \varepsilon_{a}^{-1}(x)=x+\Ow(\|x\|^{w+1})$ near $x=0$. As $\phi \circ \varepsilon_{a}^{-1}(x)= \hat{\phi}_k\circ \hat{\varepsilon}_{a}^{-1}(x)$ we see that $\hat{\phi}_k\circ \hat{\varepsilon}_{a}^{-1}(x)-x=\Ow(\|x\|^{w+1})$ near $x=0$. By Lemma~\ref{lem-eq-we} this implies that  $ \hat{\phi}\circ \hat{\varepsilon}_{a}^{-1}(x)- x_k$ has weight~$\geq w_k+1$. This is possible only if $ \hat{\phi}_k\circ \hat{\varepsilon}_{a}^{-1}(x)-x_k=0$, and so  $\hat{\phi}\circ \hat{\varepsilon}_{a}^{-1}(x)=x$, i.e., 
 $\hat{\phi}(x)= \hat{\varepsilon}_{a}(x)$. The proof is complete. 
 \end{proof}

\begin{example}[Step 1 Case]
 Assume that $r=1$, so that $H_1=H_r=TM$. Given any frame $(X_1,\ldots, X_n)$ near a some given point $a\in M$, we saw in \S\S~\ref{subsec:examples-Carnot-coord} that the linearly adapted coordinates $y=T_ax$ are Carnot coordinates. Therefore, in this case we simply have $\varepsilon_a(x)=T_ax$. 
\end{example}

\begin{example}[Heisenberg Manifolds]
 Let $(M^{d+1},H)$ be a Heisenberg manifold, where $H_1=H$ is a hyperplane bundle. Given any $H$-frame near some given point $a\in M$, it was shown in  \S\S\ref{subsec:examples-Carnot-coord} that $\psi_a=\op{id}$. It then follows that in this case we have $\varepsilon_a(x)=\phi(T_ax)$, where $\phi(x)$ is given by~(\ref{eq:Carnot-coord.phi-r2}).  
\end{example}

\begin{remark}
 In the setting of Heisenberg manifolds, the $\varepsilon$-Carnot coordinates are called anti-symmetric $y$-coordinates in~\cite{BG:CHM} and Heisenberg coordinates in~\cite{Po:Pacific1}. It was  observed in~\cite{Po:Pacific1} that in such coordinates the nilpotent approximation  at a given point $a$ is given by the tangent group $GM(a)$
\end{remark}

\subsection{$\varepsilon$-Carnot Coordinates on Graded Nilpotent Lie Groups}\label{sec:normal-group} 
Let us now focus on the construction of $\varepsilon$-Carnot coordinates on graded nilpotent Lie groups. Thus, let $G$ be a graded nilpotent Lie group of  step $r$ and dimension $n$ with unit $e$. 
Its Lie algebra $\fg=TG(e)$ then is equipped with a grading $\fg=\fg_1\oplus \cdots \oplus \fg_r$ satisfying~(\ref{eq:Carnot.grading-bracket}). 
In the same way as in Section~\ref{sec:tangent-group}, this grading gives rise to a Carnot filtration $H_1\subset \cdots \subset H_r=TG$ with
 $H_w=E_1\oplus \cdots \oplus E_w$, where $E_j$ is obtained by left-translating $\fg_j$ over $G$. In what follows, we let $(w_1,\ldots , w_n)$ be the weight sequence of $(H_1,\ldots, H_r)$. In addition, given any $a\in G$, we let $\lambda_a:G\rightarrow G$ be the left-multiplication by $a$.

\begin{definition}\label{def:Group-Carnot-coord.graded-basis}
 A \emph{graded basis} of the Lie algebra $\fg$ is any basis $(\xi_1,\ldots, \xi_n)$ such that $\xi_j \in \fg_{w_j}$ for $j=1,\ldots, n$.  
\end{definition}

 Let $(\xi_1,\ldots, \xi_n)$ be a graded basis. Then the structure constants $L_{ij}^k$ of $\fg$ with respect to this basis are such that
\begin{equation*}
  [\partial_{i},\partial_{j}]=\sum_{w_i+w_j=w_k} L_{ij}^k \xi_k. 
\end{equation*}
In particular, we see that $L_{ij}^k=0$ when $w_i+w_j\neq w_k$. For $j=1,\ldots, n$, we let $X_j$ be the left-invariant vector field on $G$ such that $X_j(e)=\xi_j$. Then $(X_1,\ldots, X_n)$ is a global $H$-frame of $TG$. In addition, as $G$ is a connected simply connected nilpotent Lie group the exponential map $\exp:\fg \rightarrow G$ is a global diffeomorphism. Therefore, we obtain a diffeomorphism $\exp_X: \R^n\rightarrow G$ by letting
\begin{equation}
 \exp_X(x)= \exp\left(x_1X_1 +\cdots + x_nX_n\right), \qquad x\in \R^n. 
 \label{eq:vareps-Carnot.canonical-coord-group}
\end{equation}
The inverse map of $\exp_X$ is a global chart. The coordinates that it defines are the usual canonical coordinates of the first kind on $G$ associated with the basis $(\xi_1,\ldots, \xi_n)$. 

Let $G^0$ be the graded nilpotent Lie group obtained by equipping $\R^n$ with the Dynkin product~(\ref{eq:GM.group-law}) associated with the structure constants $L_{ij}^k$.  In addition, for $j=1,\ldots, n$ let $X_j^0$ be the left-invariant vector field that agrees with $\partial_j$ at $x=0$. It is given by the formula~(\ref{eq:GM.Xj}) 
associated with the structure constants $L_{ij}^k$, $w_i+w_j=w_k$. It follows from their definitions that $X_1, \ldots, X_n$ satisfy the bracket relations~(\ref{eq:Nilpotent.structure-constants-Xj(a)}) with the same coefficients $L_{ij}^k$. Thus, in the notation of~(\ref{eq:Carnot-mfld.brackets-H-frame})  the coefficients $L_{ij}^k(x)$, $w_i+w_j\leq w_k$, are given by
\begin{equation}
  L_{ij}^k(x)=\left\{ 
  \begin{array}{cl}
     L_{ij}^{k} & \text{if $w_{i}+w_{j}=w_k$},\\ 
      0 & \text{if $w_{i}+w_{j}<w_k$}.
  \end{array} \right.  
  \label{eq:Charac-Carnot.structure-constants}
\end{equation}
Combining this with~(\ref{eq:GM.Xj}) we then see that, for every $a\in G$, the model vector field $X_j^a$ is given by the very same formula as that for the vector field $X_j^0$. Thus, 
\begin{equation}\label{eq-xax}
    X_{j}^{a}=X_{j}^0 \qquad \text{for  all $a\in G$ and $j=1,\ldots,n$}.
\end{equation}

\begin{proposition}\label{prop-static}
   In the canonical coordinates given by~(\ref{eq:vareps-Carnot.canonical-coord-group}) we have 
    \begin{equation*}
       \varepsilon_{a}^{X}(x)=a^{-1}\cdot x=(-a)\cdot x  \qquad \text{for all $a,x\in \R^n$},  
    \end{equation*}
    where $\cdot$ is the product law of $G^0$. 
\end{proposition}
\begin{proof}
In the same way as in the proof of Lemma~\ref{lem:Carnot-coord.admissibility-Xa}, it  can be checked that $\exp_X$ is a Lie group isomorphism from $G^0$ onto $G$.
Then $(\exp_X)_*X_j^0$ is a left-invariant vector field on $G$ such that
\begin{equation*}
 \left(\exp_X\right)_*X_j^0(e)=\left(\exp_X\right)'(0)\left[X_j^0(0)\right]=\exp_X'(0)\partial_j=\partial_j\exp_X(0). 
\end{equation*}
As $\partial_j\exp_X(0)=\left.\partial_j\exp(x_jX_j)(e)\right|_{x=0}=X_j(e)$, we see that $(\exp_X)_*X_j^0(e)=X_j(e)$. By left invariance the vector fields $(\exp_X)_*X_j^0$ and $X_j$ agree everywhere. Thus, in our canonical coordinates the $H$-frame $(X_1,\ldots, X_n)$ is identified with the frame $(X_1^0,\ldots, X_n^0)$.
We also observe that, as the product of $G^0$ is a Dynkin product, the inversion on $G^0$ is given by
\begin{equation}
 x^{-1}=-x \qquad  \text{for all $x\in G$}. 
 \label{eq:group-Carnot-coord.inversion}
\end{equation}

Given $a\in \R^n$, let $\lambda_a:\R^n\rightarrow \R^n$ be the left-multiplication by $a$ with respect to the product law of $G^0$. Then the proof only amounts to show that 
$\varepsilon_a(x)=\lambda_{-a}(x)$ for all $x\in \R^n$. Furthermore, by Theorem~\ref{thm:Carnot-coord.unicity-normal2} in order to achieve this we only have to check the following properties: 
\begin{enumerate}
 \item[(i)] The change of variables $x\rightarrow \lambda_{-a}(x)$ provides us with Carnot coordinates at $a$ adapted to $(X_1^0,\ldots, X_n^0)$. 
 
 \item[(ii)] The map $\lambda_{-a}(x)$ can be put in the form  $\hat{\varepsilon}\circ T(x)$, where $T(x)$ is an affine map such that $T(a)=0$ and $\hat{\varepsilon}(x)$ is a polynomial map of the form~(\ref{eq-phi-hat}). 
\end{enumerate}

To check (i) we note that, for $j=1,\ldots, n$, the vector field $X_{j}^0$ is left-invariant and agrees with $\partial_j$ at $x=0$. Thus, $[(\lambda_{-a})_{*}X_{j}^0](0)=X_{j}^0(0)=\partial_{j}$ for $j=1,\ldots,n$, which shows that the coordinates $y=\lambda_{-a}(x)$ are linearly adapted at $a$ to $(X_1^0,\ldots, X_n^0)$. 
Moreover, the homogeneity and left-invariance of the vector fields $X_{j}^0$ imply that
$ t^{w_{j}}\delta_{t}^{*}\left( (\lambda_{-a})_{*}X_{j}^0\right)=t^{w_{j}}\delta_{t}^{*}X_{j}^0=X_{j}^0$ for all $t\in \R^*$. 
Combining this with~(\ref{eq-xax}) and Proposition~\ref{prop:Carnot-criterion} shows that $\lambda_{-a}$ provides us with Carnot coordinates at $a$ adapted to $(X_1^0,\ldots, X_n^0)$. 

It remains to check (ii).  It follows from~(\ref{eq:GM.group-law2}) that, for every $k=1,\ldots,n$, the component 
$\lambda_{-a}(x)_{k}=[(-a)\cdot x ]_k$ is a linear combination of monomials 
$x^{\alpha}$ with $1\leq \brak \alpha\leq w_{k}$. Combining this with the fact that $\lambda_{-a}(a)=0$ we see  that the Taylor expansion of the component $\lambda_{-a}(x)_k$ at $x=a$ takes the form, 
\begin{equation}\label{eq-lax-2}
    \lambda_{-a}(x)_{k}= \lambda_{-a}'(a)_k(x-a)+R_{k}(x-a), \quad \text{where}\ R_{k}(y):=\sum_{\substack{\brak\alpha\leq 
    w_{k}\\ |\alpha|\geq 2}}\frac{1}{\alpha!}\partial^{\alpha}\lambda_{-a}(a)_ky^{\alpha}.
\end{equation}

Set $T(x)= \lambda_{-a}'(a)(x-a)$. Thanks to~(\ref{eq:group-Carnot-coord.inversion}) we have $\lambda_{-a}'(a)^{-1}= [(\lambda_a^{-1})'(a)]^{-1}=\lambda_a'(0)$. Therefore, if we set 
$\hat{R}_{k}(x):=[R_{k}\circ \lambda_{a}'(0)](x)$, then we can rewrite~(\ref{eq-lax-2}) as
\begin{equation}
 \lambda_{-a}(x)= T(x)_k+ R_k\circ \lambda_{-a}'(a)^{-1}\circ T(x)= T(x)_k+\hat{R}_k\circ T(x). 
 \label{eq:Group-Carnot-coord.lambdaaTkRk}
\end{equation}
Set $\hat{R}(x)=(\hat{R}_1(x),\ldots, \hat{R}_n(x))$ and $\hat{\varepsilon}(x)=x+\hat{R}(x)$. Then~(\ref{eq:Group-Carnot-coord.lambdaaTkRk}) shows that 
\begin{equation*}
  \lambda_{-a}(x)= \hat{\varepsilon}\circ T (x). 
\end{equation*}
By definition $T(x)$ is an affine map such that $T(a)=0$. Therefore, in order to check (ii) it only remains to show that 
the components of $\hat{\varepsilon}(x)$ are of the form~(\ref{eq-phi-hat}). 

It also follows from~(\ref{eq:GM.group-law2}) that each component $\lambda_{a}(x)_{k}=(a\cdot x)_k$ is affine with slope 1 with respect to the variable $x_k$ 
and is independent of the variables $x_{j}$ with $w_{j}\geq w_{k}$ and $j\neq k$. This implies that 
$\partial_{j}\lambda_{a}(0)_{k}=\delta_{jk}$ for $w_{j}\geq w_{k}$, and so we have
\begin{equation}\label{eq-lax-1}
    \left(\lambda_{a}'(0)x\right)_{k}=x_{k}+\sum_{w_{j}<w_{k}}\partial_{j}\lambda_{a}(0)_{k}x_{j}, \qquad k=1,\ldots, n. 
\end{equation}
Combining this with the formula for $R_k(x)$ in~(\ref{eq-lax-2}) we deduce that $\hat{R}_{k}(x)$ is a 
linear combination of monomials $x^{\alpha}$ with $\brak \alpha\leq w_{k}$ and $|\alpha|\geq 2$. Therefore, for every $k=1,\ldots, n$ 
the component $\hat{\varepsilon}_k(x)=x_k+ \hat{R}_{k}(x)$ is of the form~(\ref{eq-phi-hat}). This shows that  $\lambda_{-a}(x)$ satisfies (ii). The proof is complete. 
\end{proof}

Given any $a\in G$, let $\lambda_a:G\rightarrow G$ be the left-multiplication by $a$ on $G$. Then we have the following corollary of Proposition~\ref{prop-static}.  

\begin{corollary}\label{prop:Graded-group.Carnot-coord}
 Let $a\in G$. Then the global chart $ (\lambda_{a}\circ \exp_X)^{-1}:G\rightarrow \R^n$ provides us with $\varepsilon$-Carnot coordinates at $a$ that are adapted to the $H$-frame $(X_1,\ldots, X_n)$. 
\end{corollary}
\begin{proof}
We know by Proposition~\ref{prop-static} that the chart $\lambda_{\exp_X^{-1}(a)}^{-1} \circ \exp_X^{-1}:G\rightarrow \R^n$ provides us with $\varepsilon$-Carnot coordinates at $a$ adapted to $(X_1,\ldots, X_n)$. Furthermore, as $\exp_X$ is a group isomorphism from $G^0$ onto $G$, for all $x\in G$, we have 
\begin{equation*}
\lambda_{\exp_X^{-1}(a)}^{-1}\circ \exp_X^{-1}(x)= \left(\exp_X^{-1}(a)\right)^{-1}\cdot \exp_X^{-1}(x)=\exp_X^{-1}\left(a^{-1}\cdot x\right) 
= (\lambda_{a}\circ \exp_X)^{-1}(x). 
\end{equation*}
This gives the result. 
\end{proof}

\section{Parameter Dependence of  $\varepsilon$-Carnot Coordinates and Group Law}\label{sec:H-frame-dependence}
In this section, we give a closer look at the dependence of the $\varepsilon_{a}$-map with respect to the base 
point $a$ and the $H$-frame and compare this map to the group law of $GM(a)$ in Carnot coordinates.  

We observe that, once we are given local coordinates $x=(x_{1},\ldots,x_{n})$, the construction of the $\varepsilon$-Carnot coordinate map
$\varepsilon_{a}$ only depends on the coefficients in these local coordinates of the vector fields of the $H$-frame 
$X=(X_{1},\ldots,X_{n})$. As different choices of $H$-frames may produce different $\varepsilon$-Carnot coordinate maps, we shall denote by 
$\varepsilon_{a}^{X}$ the $\varepsilon$-Carnot coordinate map at $a$ associated with the $H$-frame $X=(X_{1},\ldots,X_{n})$. 

\subsection{Dependence on the $H$-frame}
In order to study the dependence of $\varepsilon_{a}^{X}$ with respect the $H$-frame $X=(X_{1},\ldots,X_{n})$ we shall 
denote by $B_{X}(x)$ the matrix of the coefficients of this $H$-frame in the local coordinates 
$x=(x_{1},\ldots,x_{n})$. Namely,
\begin{equation*}
    B_{X}(x)=\left( b_{jk}(x)\right), \quad \text{where } X_{j}=\sum_{1\leq k \leq n}b_{jk}(x)\partial_{x_{k}}.
\end{equation*}As it follows from the proof of Lemma~\ref{lem-affine} the affine map $T_{a}(x)$ is given by
\begin{equation*}
    T_{a}(x)=(B_{X}(a)^{t})^{-1}(x-a),
\end{equation*}where $B_{X}(a)^{t}$ is the transpose matrix of $B_{X}(a)$.

\begin{proposition}\label{prop:Carnot-prop.polynomial-X}
    For $k=1,\ldots,n$, we have
    \begin{equation}
        \varepsilon^{X}_{a}(x)_{k}=T_{a}(x)_{k}+\sum_{\substack{\brak\alpha\leq w_{k}\\|\alpha|\geq 2}} 
        c_{k\alpha}\left(B_{X}(a)\right) T_{a}(x)^{\alpha},
        \label{eq:Carnot-prop.epsX}
    \end{equation}where $T_{a}(x)=(B_{X}(a)^{t})^{-1}(x-a)$ and  each coefficient $c_{k\alpha}\left(B_{X}(a)\right)$  is a universal polynomial in the coefficients 
    of the matrices $(B_{X}(a)^{t})^{-1}$ and $\partial^{\beta}_{x}B_{X}(a)$ with $|\beta|\leq \brak\alpha-1$. 
\end{proposition}
\begin{proof}
    The proof is a consequence of a series of reductions. First, in order to distinguish between the various types of 
    coordinates at stake we set 
    \begin{equation*}
        \overline{x}=T_{a}(x), \qquad \tilde{x}=\psi_{a}(\overline{x})=\psi_{a}\circ T_{a}(x), \qquad 
        \hat{x}=\hat{\varepsilon}_{a}(\tilde{x})=\varepsilon_{a}(x).
    \end{equation*}Thus, $\overline{x}$ (resp., $\tilde{x}$, $\hat{x}$) are the linearly adapted coordinates (resp., privileged coordinates, 
    Carnot coordinates) defined by $T_{a}$ (resp., $\psi_{a}\circ T_{a}$, $\varepsilon_{a}$). In addition, for $j=1,\ldots, n$, set
    \begin{equation*}
        (T_{a})_{*}X_{j}=\sum_{1\leq j \leq n}\overline{b}_{jk}(\overline{x})\partial_{\overline{x}_{k}} \quad 
        \text{and} \quad (\psi_{a}\circ T_{a})_{*}X_{j}= \sum_{1\leq j \leq n}\tilde{b}_{jk}(\tilde{x})\partial_{\tilde{x}_{k}}.
    \end{equation*}In other words, the coefficients $\overline{b}_{jk}(\overline{x})$ (resp., 
    $\tilde{b}_{jk}(\tilde{x})$) are the coefficients of the vector fields $X_{1},\ldots, X_n$ in the coordinates $\overline{x}$ 
    (resp., $\tilde{x}$). Note also that as $T_{a}x=(B_{X}(a)^{t})^{-1}(x-a)$, the coefficients $b_{jk}(x)$ and $\overline{b}_{jk}(\overline{x})$ are related by
    \begin{equation*}
        \overline{b}_{jk}(\overline{x})=\sum_{1\leq l \leq n}b_{jl}\left(a +B_{X}(a)^{t}\overline{x}\right)(B_{X}(a)^{t})^{-1}_{kl}, \qquad j,k=1,\ldots, n.
    \end{equation*}It then follows that each partial derivative $\partial^{\alpha}_{\overline{x}}\overline{b}_{jk}(0)$ 
    is of the form, 
    \begin{equation}\label{eq-poly-2}
       \partial^{\alpha}_{\overline{x}}\overline{b}_{jk}(0)=\sum_{1\leq l \leq n}\sum_{|\beta|=\alpha} d_{jl\alpha}(B_{X}(a))
       \partial^{\beta}_{x}b_{jl}(a)(B_{X}(a)^{t})^{-1}_{kl},
    \end{equation}where the coefficients $d_{jl\alpha}(B_{X}(a))$ are polynomials in the entries of the matrix $B_{X}(a)$. 
    
    Bearing this in mind, thanks to Remark~\ref{rem-phi-hat} we know that
    \begin{equation*}
       \varepsilon_{a}(x)_{k}=T_{a}(x)_{k}+\sum_{\substack{\brak\alpha\leq w_{k}\\|\alpha|\geq 
        2}}c_{k\alpha}T_{a}(x)^{\alpha}, \quad 
        j=1,\ldots,n,
    \end{equation*}where the coefficients $c_{k\alpha}$ are as follows:
    \begin{itemize}
        \item[-] For $\brak\alpha=w_{k}$ they agree with the coefficients $c_{k\alpha}$ in~(\ref{eq:Carnot.eps-cjalpha}).  
    
        \item[-] For $\brak \alpha<w_{k}$ they are universal polynomials in the coefficients $c_{l\beta}$ with $\brak\beta=w_{l}<w_{k}$
        and the coefficients $a_{l\gamma}$ in~(\ref{eq-form-h}) with $\brak \gamma<w_{l}<w_{k}$. 
      \end{itemize}
     Combining this with~(\ref{eq-poly-2}) we see that in order to complete the proof it is enough to show that each coefficient 
    $c_{k\alpha}$ is a universal polynomial in the partial derivatives  
    $\partial^{\beta}_{\overline{x}}\overline{b}_{jl}(0)$ with $|\beta|\leq |\alpha|-1$. 
   
    Furthermore, by Remark~\ref{rem-ab} and Lemma~\ref{prop:Carnot-coord.form-phi} we know that
    \begin{itemize}
        \item[-] Each coefficient $a_{k\alpha}$, $\brak\alpha<w_k$, is a universal polynomial in the partial derivatives 
        $\partial^{\beta}_{\overline{x}}\overline{b}_{jl}(0)$ with $|\beta|\leq |\alpha|-1$.   
    
        \item[-] Each coefficient $c_{k\alpha}$, $\brak\alpha=w_{k}$, is a universal polynomial in the partial 
        derivatives $\partial_{\tilde{x}}^{\beta}\tilde{b}_{jl}(0)$ with $ \brak\beta+w_j=w_{l}\leq w_k$.   
    \end{itemize}
   We note that if $|\beta|\leq |\alpha|-1$, then $|\beta|\leq \brak\alpha-1$. Moreover, if $\brak\alpha=w_{k}$ and $ \brak\beta+w_j=w_{l}\leq w_k$, then $|\beta|\leq \brak \beta=w_l-w_j\leq w_k-1=\brak\alpha -1$. Therefore, in order to complete the proof we only have to prove that each partial derivative 
    $\partial_{\tilde{x}}^{\alpha}\tilde{b}_{jk}(0)$ is a universal polynomial in the partial derivatives 
    $\partial^{\beta}_{\overline{x}}\overline{b}_{lp}(0)$ with $|\beta|\leq |\alpha|$.
    To see this we note that the coefficients $\tilde{b}_{jk}(\tilde{x})$ and $\overline{b}_{jk}(\overline{x})$ are 
    related by
    \begin{equation*}
        \tilde{b}_{jk}(\tilde{x})=\sum_{1\leq l\leq n}\overline{b}_{jl}\left(\psi_{a}^{-1}(x)\right) 
       ( \partial_{\overline{x}_{l}}\psi_{a})\left(\psi_{a}^{-1}(x)\right)_{k}.
    \end{equation*}We then deduce that each partial derivative  $\partial_{\tilde{x}}^{\alpha}\tilde{b}_{jk}(0)$ is a 
    universal polynomial in the following
    \begin{itemize}
        \item[-] The partial derivatives $\partial^{\beta}_{\overline{x}}b_{jl}(0)$ with $|\beta|\leq |\alpha|$.  
    
        \item[-] The partial derivatives $\partial^{\gamma}_{\overline{x}}\psi_{a}(0)_{k}$ with $|\gamma|\leq 
        |\alpha|+1$. 
    
        \item[-] The partial derivatives $\partial^{\gamma}_{\tilde{x}}\psi_{a}^{-1}(0)_{k}$ with $|\gamma|\leq 
        |\alpha|$. 
    \end{itemize}
   This further reduces the proof to showing that the partial derivative 
   $\partial^{\alpha}_{\overline{x}}\psi_{a}(0)_{k}$ and $\partial^{\alpha}_{\overline{x}}\psi_{a}^{-1}(0)_{k}$ are 
   universal polynomials in the partial derivatives $\partial^{\beta}_{\overline{x}}\overline{b}_{jl}(0)$ with 
   $|\beta|\leq |\alpha|-1$. 
   
   Having said this, the partial derivatives $\partial^{\alpha}_{\overline{x}}\psi_{a}(0)_{k}$ can be computed 
   explicitly. Indeed, we have 
   \begin{equation}\label{eq-poly-1}
       \psi_{a}(\overline{x})_{k}=\overline{x}_{k}+\sum_{\substack{\brak\alpha<w_{k}\\|\alpha|\geq 
       2}}a_{k\alpha}\overline{x}^{\alpha}, \qquad k=1,\ldots,n.
   \end{equation}Thus $\psi_{a}'(0)=\op{id}$ and, when $|\alpha|\geq 2$, we have
   \begin{equation*}
       \partial^{\alpha}_{\overline{x}}\psi_{a}(0)_{k}=\left\{
       \begin{array}{ll}
           \alpha!a_{k\alpha} & \text{if $\brak\alpha<w_{k}$}\\
           0 & \text{otherwise}.
        \end{array}\right.
   \end{equation*}
   As $a_{k\alpha}$ is a universal polynomial in the partial derivatives $\partial^{\beta}_{\overline{x}}\overline{b}_{jl}(0)$ with 
   $|\beta|\leq |\alpha|-1$, we then deduce that we have the same property with $\partial^{\alpha}_{\overline{x}}\psi_{a}(0)_{k}$. 
   In addition, by using~(\ref{eq-poly-1}) and arguing by induction as in the proof of Theorem~\ref{thm:Carnot-coord.unicity-normal2} it can be shown that the components 
   $\psi_{a}^{-1}(\tilde{x})_k$, $k=1,\ldots, n$,  are of the form,
   \begin{equation*}
      \psi_{a}^{-1}(\tilde{x})_{j}= \tilde{x}_{j}+\sum_{\substack{\brak\alpha<w_{k}\\|\alpha|\geq 
      2}}\tilde{a}_{k\alpha} \tilde{x}^{\alpha},
   \end{equation*}where $\tilde{a}_{k\alpha}$ is a universal polynomial in the coefficients $a_{l\beta}$ with 
   $w_{l}\leq w_{k}$ and $|\beta|\leq |\alpha|$. Therefore, it is a universal polynomial in the partial derivatives $\partial^{\beta}_{\overline{x}}\overline{b}_{jl}(0)$ with 
   $|\beta|\leq |\alpha|-1$. By arguing as above we then can show that each partial derivative 
   $\partial^{\alpha}_{\tilde{x}}\psi_{a}^{-1}(0)_{k}$ is a universal polynomial in the partial derivatives $\partial^{\beta}_{\overline{x}}\overline{b}_{jl}(0)$ with 
   $|\beta|\leq |\alpha|-1$. This completes the proof.  
\end{proof}

The following  is  a version of Proposition~\ref{prop:Carnot-prop.polynomial-X} for the inverse map 
$(\varepsilon_{a}^{X})^{-1}(x)$. 

\begin{proposition}\label{polynomial-XI}
    The inverse map $(\varepsilon_{a}^{X})^{-1}(x)$ is of the form,
    \begin{gather}
    (\varepsilon_{a}^{X})^{-1}(x)= a+ B_{X}(a)^{t}\check{\varepsilon}_{a}(x),  \nonumber \\
       \check{\varepsilon}_{a}(x)_{k}=x_{k}+\sum_{\substack{\brak\alpha\leq w_{k}\\|\alpha|\geq 2}} 
       \check{c}_{k\alpha}\left(B_{X}(a)\right) x^{\alpha}, \quad k=1,\ldots,n, 
       \label{eq:Carnot-prop.epsXI}
    \end{gather}where $ \check{c}_{k\alpha}(B_{X}(a))$ is a universal polynomial in the entries 
    of the matrices $(B_{X}(a)^{t})^{-1}$ and $\partial^{\beta}_{x}B_{X}(a)$ with $|\beta|\leq \brak\alpha-1$. 
\end{proposition}
\begin{proof}
    We know by Proposition~\ref{prop:Carnot-prop.polynomial-X} that $\varepsilon_{a}^{X}(x)=\hat{\varepsilon}_{a}\circ T_{a}(x)$, where
    $T_{a}(x)=(B_{X}(a)^{t})^{-1}(x-a)$ and $\hat{\varepsilon}_{a}(x)$ is of the form, 
        \begin{equation}\label{eq-hat-1}
        \hat{\varepsilon}_{a}(x)_{k}=x_{k}+\sum_{\substack{\brak\alpha\leq w_{k}\\|\alpha|\geq 2}} 
        c_{k\alpha}\left(B_{X}(a)\right) x^{\alpha}, \qquad k=1,\ldots,n,
    \end{equation}where $c_{k\alpha}(B_{X}(a))$ is a universal polynomial in the entries 
    of the matrices $(B_{X}(a)^{t})^{-1}$ and $\partial^{\beta}_{x}B_{X}(a)$ with $|\beta|\leq |\alpha|-1$. Set 
    $\check{\varepsilon}_{a}=(\hat{\varepsilon}_{a})^{-1}$. Then we have
    \begin{equation*}
         (\varepsilon_{a}^{X})^{-1}(x)=T_{a}^{-1}\circ\check{\varepsilon}_{a}(x)=a+B_{X}(a)^{t}\check{\varepsilon}_{a}(x).
    \end{equation*}
    
    Moreover, in the same way as in the proof of Theorem~\ref{thm:Carnot-coord.unicity-normal2}, if in~(\ref{eq-hat-1}) 
    we substitute $\check{\varepsilon}_{a}(x)$ for $x$ and solve for $x_k$, then we get  
     \begin{equation*}
         \check{\varepsilon}_{a}(x)_k= x_k-  \sum_{\substack{\brak\alpha\leq w_{k}\\|\alpha|\geq 2}} 
        c_{k\alpha}\left(B_{X}(a)\right)  \prod_{w_l<w_k}\check{\varepsilon}_{a}(x)_l^{\alpha_l}, \qquad k=1,\ldots,n. 
     \end{equation*}
    When $w_k=1$ this gives $ \check{\varepsilon}_{a}(x)_k= x_k$. In general, an induction on $w_k$ shows that $ \check{\varepsilon}_{a}(x)_k$ is of the 
    form~(\ref{eq:Carnot-prop.epsXI}), where each coefficient $ \check{c}_{k\alpha}(B_{X}(a))$ is a universal polynomial in the entries 
    of $(B_{X}(a)^{t})^{-1}$ and $\partial^{\beta}_{x}B_{X}(a)$ with $|\beta|\leq \brak\alpha-1$. The result is 
    thus proved. 
\end{proof}

As immediate consequences of Proposition~\ref{prop:Carnot-prop.polynomial-X} and Proposition~\ref{polynomial-XI} we obtain the following smoothness result. 

\begin{corollary}\label{prop-carnot-sm}
    Given an $H$-frame $X=(X_{1},\ldots,X_{n})$ over a coordinate open $U\subset \R^{n}$, the maps $(y,x)\rightarrow 
    \varepsilon_{x}(y)\in \R^{n}$ and $(y,x)\rightarrow 
    \varepsilon_{x}^{-1}(y)$ are smooth maps from $U\times \R^{n}$ to $\R^{n}$. 
\end{corollary}

Let $U$ be an open subset of $\R^{n}$ and denote by $\cH(U)$ the subspace of $\cX(U)^{n}$ consisting of frames 
$(X_{1},\ldots,X_{n})$ such that, for $i,j=1,\ldots,n$ with $w_{i}+w_{j}\leq r$, we have
\begin{equation*}
    [X_{i},X_{j}](x)\in \op{Span}\left\{X_{k}(x);\ w_{k}\leq w_{i}+w_{j}\right\} \qquad \text{for all $x\in U$}.
\end{equation*}
We note that if $X=(X_{1},\ldots,X_{n})$ is in 
$\cH(U)$, then this is an $H$-frame with respect to the filtration $H=(H_{1},\ldots,H_{r})$, where $H_{w}$ is the 
sub-bundle generated by the vector fields $X_{j}$, $w_{j}\leq w$. 

We equip $\cH(U)$ with the topology induced by that of $\cX(U)^{n}$. We observe that 
in~(\ref{eq:Carnot-prop.epsX}) and~(\ref{eq:Carnot-prop.epsXI}) the 
coefficients $c_{k\alpha}\left(B_{X}(a)\right)$ and $\check{c}_{k\alpha}\left(B_{X}(a)\right)$ depend continuously on 
$X=(X_{1},\ldots,X_{n})$ with respect to the aforementioned topology on $\cH(U)$. Therefore, we arrive 
at the following statement. 

\begin{corollary}\label{prop:Carnot-prop.eps-continuity}
    The maps $X\rightarrow \varepsilon^{X}_{y}(x)$ and $X\rightarrow \left(\varepsilon^{X}_{y}\right)^{-1}(x)$ are continuous maps from 
    $\cH(U)$ to $C^{\infty}(U\times \R^{n}, \R^{n})$. 
\end{corollary}

Finally, the following shows that $\varepsilon$-Carnot coordinates behaves nicely under anisotropic rescalings of $H$-frames. 

\begin{proposition}\label{prop:data-dependence.rescaling-frame}
     Let  $(x_{1},\ldots,x_{n})$ be local coordinates near $a$ with range $U\subset \R^{n}$. Let $t\in \R^*$ and $y\in U$ be such that 
     $t\cdot y\in U$. In addition, for $j=1,\ldots,n$, set $\hat{X}_{j}=t^{w_{j}}\delta_{t}^{*}X_{j}$. Then
     \begin{equation*}
       \varepsilon_{y}^{\hat{X}}(x)=  t^{-1}\cdot  \varepsilon_{t\cdot y}^{X}(t\cdot x) \qquad \forall x\in 
        \R^{n}.
     \end{equation*}
\end{proposition}
\begin{proof}
 Set $\phi=\delta_t^{-1}\circ \varepsilon_{t\cdot y}^{X}\circ \delta_t$ and $V=\phi(U)$. The proof amounts to show that $\phi=  \varepsilon_{y}^{\hat{X}}$. By Proposition~\ref{prop:Carnot-criterion} and Theorem~\ref{thm:Carnot-coord.unicity-normal2} to achieve this we only have to check the following three properties: 
 \begin{enumerate}
\item[(i)] We can put $\phi$ in the form $\phi=\hat{\phi}\circ T$, where $T$ an affine map such that $T(y)=0$ and $\hat\phi(x)$ is a polynomial map of the form~(\ref{eq-phi-hat}). 

\item[(ii)] The coordinates provided by $\phi$ are linearly adapted  to $(\hat{X}_1, \ldots, \hat{X}_n)$ at $y$. 

\item[(iii)] Given any $s\in \R^*$, for $j=1,\ldots, n$ and as $t\rightarrow 0$, we have
\begin{equation*}
 s^{w_j}\delta_s^*\left( \phi_*X_j\right) = \hat{X}_j^y +\op{O}(t) \qquad \text{in $\cX(V)$}. 
\end{equation*}
\end{enumerate}

We know that $\varepsilon^X_{t\cdot y}=\varepsilon\circ T$, where $T$ is an affine map such that $T(t\cdot y)=0$ and $\hat\varepsilon$ is a polynomial map of the form~(\ref{eq-phi-hat}). Thus, we have $\phi= \delta_t^{-1}\circ( \varepsilon\circ T)\circ \delta_t=  (\delta_t^{-1}\circ \varepsilon\circ \delta_t)\circ (\delta_t^{-1}\circ T\circ \delta_t)$. 
As the dilations $\delta_t$ and $\delta_t^{-1}$ are diagonal linear maps, we see that $\delta_t^{-1}\circ T\circ \delta_t$ is an affine map such that $\delta_t^{-1}\circ T\circ \delta_t(y)=\delta_t^{-1}\circ T(t\cdot y)=0$ and $\delta_t^{-1}\circ \varepsilon\circ \delta_t$ is a polynomial map of the form~(\ref{eq-phi-hat}). Therefore, the property (i) is satisfied. 

To check (ii) and (iii) we observe that, for $j=1,\ldots, n$, we have 
\begin{equation}
 \phi_*\hat{X}_j= \left( \delta_t^{-1}\circ \varepsilon_{t\cdot y}^{X}\circ \delta_t\right)_*\left( t^{w_{j}}\delta_{t}^{*}X_{j}\right) = t^{w_j} \delta_t^* \left( (\varepsilon_{t\cdot y}^{X})_*X_j\right). 
 \label{eq:data-dependence.phi8hatX}
\end{equation}
The coordinates provided by $ \varepsilon_{t\cdot y}^{X}$ are Carnot coordinates at $t\cdot y$ that are adapted to $(X_1,\ldots, X_n)$. In particular, they are linearly adapted coordinates, and so $(\varepsilon_{t\cdot y}^{X})_*X_j(0)=\partial_j$. Combining this with~(\ref{eq:anisotropic-dtX}) and~(\ref{eq:data-dependence.phi8hatX}) 
we obtain
\begin{equation*}
 \phi_*\hat{X}_j(0)= t^{w_j} \delta_t^* \left( (\varepsilon_{t\cdot y}^{X})_*X_j\right)(0) = t^{w_j} t^{-w_j} \partial_j= \partial_j. 
\end{equation*}
Thus, the coordinates provided by $\phi$ are linearly adapted  to $(\hat{X}_1, \ldots, \hat{X}_n)$ at $y$, i.e., we have (ii). 

To complete the proof it remains to check that (iii) is satisfied. Using~(\ref{eq:data-dependence.phi8hatX}) we see that, for any $s\in \R^*$ and $j=1,\ldots, n$, we have 
\begin{equation*}
 s^{w_j} \delta_s^* \left( \phi_*X_j\right)=  s^{w_j}t^{w_j} \delta_s^*  \delta_t^* \left( (\varepsilon_{t\cdot y}^{X})_*X_j\right)= 
  (st)^{w_j}  \delta_{st}^* \left( (\varepsilon_{t\cdot y}^{X})_*X_j\right). 
\end{equation*}
Recall that $\varepsilon_{t\cdot y}^{X}$ provides us with Carnot coordinates at $t\cdot y$ that are adapted to $(X_1,\ldots, X_n)$. Therefore, by using Proposition~\ref{prop-pri-equiv} we see that, as $s\rightarrow 0$ and in $\cX(V)$, we have 
\begin{equation*}
  s^{w_j} \delta_s^* \left( \phi_*X_j\right)=  (st)^{w_j}  \delta_{st}^* \left( (\varepsilon_{t\cdot y}^{X})_*X_j\right) = X_j^{(t\cdot y)} +\op{O}(t). 
\end{equation*}
Therefore, in order to check that we have (iii) we only need to show that $ X_j^{(t\cdot y)}=\hat{X}_j^{y}$ for $j=1,\ldots, n$. 

In the notation of~(\ref{eq:Carnot-mfld.brackets-H-frame}) there are smooth functions $L^k_{ij}(x)$ near $t\cdot y$ and smooth functions $\hat{L}^k_{ij}(x)$ near $y$ such that
\begin{equation*}
 [X_j,X_j]=\sum_{w_k\leq w_i+w_j} L_{ij}^k(x) X_k \quad \text{and} \quad  [\hat{X}_j,\hat{X}_j]=\sum_{w_k\leq w_i+w_j} \hat{L}_{ij}^k(x) \hat{X}_k. 
\end{equation*}
As $ [\hat{X}_j,\hat{X}_j]= t^{w_i+w_j} [\delta_{t}^{*}X_{i}, \delta_{t}^{*}X_{j}]=  t^{w_i+w_j}\delta_t^*([X_i,X_j])$, we get
\begin{equation*}
 [\hat{X}_j,\hat{X}_j]=  t^{w_i+w_j} \!\!\! \sum_{w_k\leq w_i+w_j} \delta_t^*\left( L_{ij}^k X_k  \right) = \!\!\sum_{w_k\leq w_i+w_j} t^{w_i+w_j-w_k} L_{ij}^k(t\cdot x) 
  \hat{X_k}. 
\end{equation*}
 Thus, near $x=y$, we have 
 \begin{equation}
 \hat{L}_{ij}^k(x) =  t^{w_i+w_j-w_k} L_{ij}^k(t\cdot x), \qquad w_k\leq w_i+w_j. 
 \label{eq:data-dependence.hatL-L}
\end{equation}
Recall that, for $j=1,\ldots, n$, the vector field $X_j^{(t\cdot y)}$ (resp., $\hat{X}_j^{y}$) is the vector field~(\ref{eq:GM.Xj}) associated with the structure constants 
$L_{pq}^l(t\cdot y)$ (resp., $\hat{L}_{pq}^l(y)$) with $w_p+w_q=w_l$. As~(\ref{eq:data-dependence.hatL-L}) implies that $L_{pq}^l(t\cdot y)=\hat{L}_{pq}^l(y)$ when $w_p+w_q=w_l$, we deduce that $ X_j^{(t\cdot y)}=\hat{X}_j^{y}$ for every $j=1,\ldots, n$. This shows that the property (iii) is satisfied. The proof is complete. 
\end{proof}

\subsection{$\varepsilon$-Carnot coordinates and group law} 
As an application of Corollary~\ref{prop:Carnot-prop.eps-continuity}, we shall now explain how well the  
$\varepsilon$-Carnot coordinate map $(x,y)\rightarrow \varepsilon_y^X(x)$ is approximated by the product of $GM(a)$ in Carnot coordinates. In particular, this will provide us with an asymptotic version of Proposition~\ref{prop-static} for general Carnot manifolds. 

Let $(X_1,\ldots, X_n)$ be an $H$-frame near $a$. As mentioned in Remark~\ref{rmk:Tangent-group.frame-fgM}, this frame gives rise to a graded basis $(\xi_1(a), \ldots, \xi_n(a))$ of $\fg M(a)$, where $\xi_j(a)$ is the class of $X_j(a)$ in $\fg_{w_j} M(a)$. By Proposition~\ref{prop:Carnot-coord.GM(a)-G(a)} this graded basis defines a global system of coordinates on $\fg M(a)=GM(a)$ in which the tangent group $GM(a)$ is identified with the graded nilpotent  group $G(a)$. This group is $\R^n$ equipped with the Dynkin product~(\ref{eq:GM.group-law}) associated with the structure constants $L_{ij}^k(a)$, $w_i+w_j=w_k$, in~(\ref{eq:Carnot-mfld.brackets-H-frame}). 

In addition, we endow $\R^n\times \R^n$ with the dilations, 
\begin{equation}
 t \cdot (x,y)=(t\cdot x, t\cdot y), \qquad x,y\in \R^n,\ t\in \R. 
 \label{eq:smoothness.dilations-2n}
\end{equation}
We also assume that we are given a pseudo-norm $\| \cdot \|:\R^n\times \R^n \rightarrow [0,\infty)$ that satisfies~(\ref{eq:anisotropic.homogeneity-pseudo-norm}) with respect to these dilations. For instance, we may take 
 \begin{equation*}
\|(x,y)\|= |x_1|^{\frac{1}{w_1}} + |y_1|^{\frac{1}{w_1}} + \cdots + |x_n|^{\frac{1}{w_n}}+ |y_n|^{\frac{1}{w_n}},
 \qquad x,y\in \R^n.
\end{equation*}
  The pseudo-norm  $\| \cdot \|$ on $\R^n\times \R^n$ enables us to speak about $\Ow(\|(x,y)\|^{w+m})$-maps in the same way as in Definition~\ref{def:anisotropic.Thetaw}. In particular, if $\Theta(x,y)=\left(\Theta_{1}(x,y),\ldots, \Theta_{n}(x,y)\right)$ is a smooth map from 
    $U$ to $\R^{n}$, where $U$ is an open neighborhood 
    of $(0,0)\in \R^{n}\times \R^{n}$, then $\Theta(x,y)=\Ow( \|(x,y)\|^{w+m})$ near $(x,y)=(0,0)$ if and only if 
    $ \Theta_{j}(x,y)=\op{O}(\|(x,y)\|^{w_{j}+m})$ for $j=1,\ldots, n$. In addition,  Lemma~\ref{lem-eq-we} provides us with
     the following characterization of $\Ow( \|(x,y)\|^{w+m})$-maps. 

\begin{lemma}\label{lem-smooth-2}
    Let $\Theta(x,y)=(\Theta_{1}(x,y),\ldots,\Theta_{n}(x,y))$  be a smooth map from $W$ to $\R^{n}$, where $W$ is some open neighborhood 
    of $(0,0)\in \R^{n}\times \R^{n}$. 
 Then the following are equivalent:
 \begin{enumerate}
     \item[(i)]   $\Theta(x,y)=\Ow(\|(x,y)\|^{w+m})$ near $(x,y)=(0,0)$.
 
     \item[(ii)]  For all $x,y\in \R^{n}$, we have $t^{-1}\cdot \Theta(t\cdot x,t\cdot y)=\op{O}(t^{m})$ as $t\rightarrow 0$. 
 \end{enumerate}
\end{lemma}

We are now in a position to derive the following approximation result.  

\begin{proposition}\label{prop-com}
    Let $(x_{1},\ldots,x_{n})$ be Carnot coordinates at $a$ that are adapted to the $H$-frame $(X_{1},\ldots,X_{n})$. Then, near 
    $(x,y)=(0,0)$, we have
    \begin{gather}
        \varepsilon_{y}^{X}(x)=(-y)\cdot x +\Ow\left(\|(x,y)\|^{w+1}\right),
        \label{eq:smoothness.vareps-yx1} \\ 
        \left(\varepsilon_{y}^{X}\right)^{-1}(x)=y\cdot x +\Ow\left(\|(x,y)\|^{w+1}\right), 
               \label{eq:smoothness.vareps-yx2}
    \end{gather}where $\cdot$ is the group law of $G(a)$ (i.e., the group law of $GM(a)$ under the identification described above). 
\end{proposition}
\begin{proof}
 Let us denote by $U$ the range of the local coordinates $(x_{1},\ldots,x_{n})$. In addition, for $t\neq 0$ and 
 $j=1,\ldots,n$, set $\hat{X}_{j}^{t}=t^{w_{j}}\delta_{t}^{*}X_{j}$. Let $x\in \R^{n}$ and $y\in \R^{n}$ be such 
 that $t\cdot y \in  U$. Then by 
 Proposition~\ref{prop:data-dependence.rescaling-frame} we have
 \begin{equation*}
   t^{-1}\cdot  \varepsilon_{t\cdot y}^{X}(t\cdot x)=  \varepsilon_{y}^{\hat{X}^{t}}(x), 
 \end{equation*}where $\varepsilon_{y}^{\hat{X}^{t}}(x)$ is the $\varepsilon$-Carnot coordinate map associated with 
 $(\hat{X}_{j}^{t},\ldots, \hat{X}_{n}^{t})$. 
 Moreover, as $(x_{1},\ldots,x_{n})$ provides us with Carnot coordinates at $a$ adapted to 
 $(X_{1},\ldots,X_{n})$, we know by Proposition~\ref{prop:Carnot-criterion} that,  for $j=1,\ldots,n$ and as $t\rightarrow 0^{+}$, we have
 \begin{equation*}
  \hat{X}_{j}^{t}= t^{w_{j}}\delta_{t}^{*}X_{j}=X_{j}^{a}+\op{O}(t) \qquad \text{in $\cX(U)$}.
 \end{equation*}
 Combining this with Corollary~\ref{prop:Carnot-prop.eps-continuity} we deduce that, as $t\rightarrow 0$, we have
 \begin{equation*}
    \varepsilon_{y}^{\hat{X}^{t}}(x)=\varepsilon_{y}^{X^{a}}(x)+ \op{O}(t) \qquad \text{in 
    $C^{\infty}(U,\R^{n})$},
 \end{equation*}where $\varepsilon_{y}^{X^{a}}$ is the $\varepsilon$-Carnot coordinate map associated with $y$ and the frame 
 $(X_{1}^{a},\ldots,X_{n}^{a})$. 
By Lemma~\ref{lem:Carnot-coord.admissibility-Xa} we know that $\exp_{X^a}$ is the identity map. Therefore, by using Proposition~\ref{prop-static} we see that $\varepsilon_{y}^{X^{a}}(x)=(-y)\cdot x$, and so for all $(x,y)\in \R^{n}\times \R^{n}$, we have
 \begin{equation*}
    t^{-1}\cdot  \varepsilon_{t\cdot y}^{X}(t\cdot x)=  \varepsilon_{y}^{\hat{X}^{t}}(x)=(-y)\cdot x + \op{O}(t) \quad \text{as $t\rightarrow 
       0^{+}$}. 
 \end{equation*}
Combining this with Lemma~\ref{lem-smooth-2} gives the asymptotics~(\ref{eq:smoothness.vareps-yx1}). 

It remains to prove that $ (\varepsilon_{y}^{X})^{-1}(x)= y\cdot x +\Ow(\|(x,y)\|^{w+1})$. Let $\Phi: \R^n\times U \rightarrow \R^n\times U $ be the smooth map defined by
\begin{equation*}
\Phi(x,y)=\left( \varepsilon_{y}^{X}(x),y\right) \qquad \text{for all $(x,y)\in \R^n\times U$}. 
\end{equation*}
 This is a diffeomorphism whose inverse map is given by
  \begin{equation}
\Phi^{-1}(x,y)= \left( \left(\varepsilon_{y}^{X}\right)^{-1}(x),y\right), \qquad (x,y)\in \R^n\times U. 
\label{eq:smoothness.Phi-inverse}
\end{equation}
Moreover, the asymptotic expansion~(\ref{eq:smoothness.vareps-yx1}) implies that, near $(x,y)=(0,0)$, we have 
 \begin{equation}
\Phi(x,y)=\left((-y)\cdot x, y\right) +\Ow\left(\|(x,y)\|^{w+1}\right). 
\label{eq:smoothness.Phi-Ow}
\end{equation}

Set $\hat{\Phi}(x,y)=((-y)\cdot x, y)$, $x,y\in \R^n$. Then $\hat{\Phi}$ is a polynomial map which is $w$-homogeneous with respect to the dilations~(\ref{eq:smoothness.dilations-2n}). Moreover, as $-y$ is the inverse of $y$ with respect to the product law of $GM(a)$, we also see that $\hat\Phi(x,y)$ is a diffeomorphism with inverse $\hat\Phi^{-1}(x,y)=(y\cdot x, y)$. Therefore, we may combine~(\ref{eq:smoothness.Phi-Ow}) and Proposition~\ref{lem:Carnot-coord.inverse-Ow} to see that, near $(x,y)=(0,0)$, we have 
\begin{equation*}
\Phi^{-1}(x,y)= \hat\Phi^{-1}(x,y) +\Ow\left(\|(x,y)\|^{w+1}\right) = (y\cdot x, y) +\Ow\left(\|(x,y)\|^{w+1}\right). 
 \end{equation*}
 Combining this with~(\ref{eq:smoothness.Phi-inverse}) then shows that $ (\varepsilon_{y}^{X})^{-1}(x)= y\cdot x +\Ow(\|(x,y)\|^{w+1})$. The proof is complete. 
\end{proof}

\begin{remark}
Proposition~\ref{prop-com} is a key ingredient in the construction of tangent groupoid of a Carnot manifold in~\cite{CP:Groupoid}. It is also an important ingredient in the construction of a full symbolic calculus for hypoelliptic pseudodifferential operators on Carnot manifolds in~\cite{CP:Carnot-calculus}. 
\end{remark}

\end{document}